\documentclass[12pt,a4paper]{report}

\usepackage[utf8]{inputenc}
\usepackage[T1]{fontenc}

\usepackage{amsmath}
\usepackage{bbm}
\usepackage{amsfonts}
\usepackage{amssymb}
\usepackage[colorlinks=true, pdfstartview=FitV, linkcolor=blue, citecolor=blue, urlcolor=blue,pagebackref=false]{hyperref}
\usepackage[usenames]{color}
\usepackage{todonotes}
\usepackage{xcolor}
\usepackage{comment}
\usepackage{enumitem}
\usepackage{hyperref}
\usepackage{graphicx}

\usepackage{enumerate}
\usepackage{tikz}
\usetikzlibrary{arrows}
\usetikzlibrary{decorations.markings}

\makeatletter
\let\orgdescriptionlabel\descriptionlabel
\renewcommand*{\descriptionlabel}[1]{%
  \let\orglabel\label
  \let\label\@gobble
  \phantomsection
  \edef\@currentlabel{#1}%
  \let\label\orglabel
  \orgdescriptionlabel{#1}%
}
\makeatother

\usepackage{amsthm}

\newtheorem{theorem}{Theorem}[chapter]
\newtheorem{corollary}[theorem]{Corollary}
\newtheorem{conjecture}[theorem]{Conjecture}
\newtheorem{proposition}[theorem]{Proposition}
\newtheorem{lemma}[theorem]{Lemma}

\theoremstyle{definition}
\newtheorem{definition}[theorem]{Definition}

\newtheorem{example}[theorem]{Example}

\theoremstyle{remark}
\newtheorem{remark}[theorem]{Remark}

\newcommand*{\grad}{\textup{grad }}
\newcommand*{\dive}{\textup{div }}
\newcommand*{\dete}{\textup{det }}
\newcommand*{\dif}[1]{\frac{\partial}{\partial x_{#1}} \bigg\rvert_x}
\newcommand*{\diff}[1]{\frac{\partial f}{\partial x_{#1}}}
\title{Combinatorial, Bakry-\'{E}mery, Ollivier's Ricci curvature notions and their motivation from Riemannian geometry}
\author{Supanat Kamtue}
\date{March, 2018}

\begin{document}

\begin{titlepage}
    \begin{center}
        \vspace*{1cm}
        
        \Huge
        \textbf{Combinatorial, Bakry-\'{E}mery, Ollivier's Ricci curvature notions and their motivation from Riemannian geometry}
        
        \vspace{0.5cm}
        \LARGE
        {}
        
        \vspace{1.5cm}
        
        \textbf{Supanat Kamtue}
        
        \vfill
        
%A dissertation submitted in partial fulfilment of the requirements for admission to the degree of MSc in Mathematical Sciences at Durham University
        
        %\vspace{0.8cm}

        \Large
        Department of Mathematical Sciences\\
        Durham University, UK\\
        %submitted: September, 2017\\
        %revisited: March, 2018
    \end{center}
\end{titlepage}
\newpage
% Add the title section.
%\maketitle

\begin{abstract}
In this survey, we study three different notions of curvature that are defined on graphs, namely, combinatorial curvature, Bakry-\'Emery curvature, and Ollivier's Ricci curvature. For each curvature notion, the definition and its motivation from Riemannian geometry will be explained.  Moreover, we bring together some global results and geometric concepts in Riemannian geometry that are related to curvature (e.g. Bonnet-Myers theorem, Laplacian operator, Lichnerowicz theorem, Cheeger constant), and then compare them to the discrete analogues in some (if not all) of the discrete curvature notions. The structure of this survey is as follows: the first chapter is dedicated to relevant background in Riemannian geometry. Each following chapter is focussing on one of the discrete curvature notions. This survay is an MSc dissertation in Mathematical Sciences at Durham University.

\end{abstract}

\tableofcontents

\chapter{Background in Riemannian Geometry}
In this chapter, we provide substantial background material from Riemannian geometry, which will prepare the readers to be able to compare to discrete analogues in graphs in later chapters. First in Section \ref{GBCC}, we introduce the Gauss-Bonnet theorem, Cartan-Hadamard theorem, and Cheeger constant, which are three examples of global concepts of manifolds that can also be illustrated as geometric features in graphs as we will see in combinatorial curvature in Chapter \ref{CCchap}. Next in Section \ref{LBsec}, we consider linear operators on manifolds including gradient, divergence, Laplacian, and Hessian. They are ingredients in Bochner's formula, which is the main motivation for Bakry-\'{E}mery curvature in Chapter \ref{BECchap}. In Section \ref{eigenlap}, the crucial operator, Laplacian, and its smallest eigenvalue have been investigated in Lichnerowicz Theorem. In Section \ref{BMsec}, we state and prove the theorem of Bonnet-Myers. In Section \ref{avgsec}, we explain the problem of finding average distance of two balls, which motivates Ollivier's Ricci curvature in Chapter \ref{ORCchap}. Lastly in Section \ref{expsec}, we give examples of manifolds and their representing graphs, and then discuss about their curvature in different notions.

\section{Gauss-Bonnet, Cartan-Hadamard, and \\ Cheeger constant} \label{GBCC}

The purpose of this section is to present theorems about curvature in Riemannian geometry, which will be compared to the discrete analogues in combinatorial curvature in Chapter \ref{CCchap}. The content of this section is divided into two parts. In the first half, we introduce (without proof) \textup{Gauss-Bonnet theorem} and \textup{Cartan-Hadamard theorem}. In the second half, we give the definition of \textup{Cheeger isoperimetric constant} (or in short, \textup{Cheeger constant}), and give the statements and sketches of proof for another two theorems that are related to Cheeger constant.

\subsection*{Gauss-Bonnet and Cartan-Hadamard}
Gauss-Bonnet theorem states that, for any closed surface (i.e. a compact two-dimensional manifold without boundary), its total curvature is equal to its Euler's characteristic multiplied by $2\pi$. A proof of this theorem can be found in e.g. \cite[pp.~274--276]{Carmo2}.
\begin{theorem}[Gauss-Bonnet]
Let $M$ be a closed surface. Then its total curvature is
$$\int\limits_{M} K \textup{d}A = 2\pi \chi(M)$$ 
\noindent where $K$ is Gaussian curvature, and $\textup{d}A$ is the area element, and $\chi(M)$ is the Euler's characteristic of $M$.
\end{theorem}

Euler's characteristic is a global topological invariant of a surface. In particular, if $M$ is orientable then $\chi(M)=2-2g$, where $g$ is a genus of $M$.  For example, a two-dimensional sphere of radius $r$ has Gaussian curvature equal to $r^{-2}$ everywhere, and its surface area is $4\pi r^2$. Hence the total curvature is equal to $\int\limits_{S^2} K \textup{d}A = 4\pi = 2\pi\chi(S^2)$. 

While Gauss-Bonnet theorem mentions the total curvature of manifolds, many other theorems (e.g. Bonnet-Myers, and Lichnerowicz) refers to properties of manifolds that have the same sign of curvature everywhere. Among those theorems, Cartan-Hadamard theorem gives an implication when a manifold has non-positive sectional curvature everywhere. The statement of the theorem is given as follows, and a proof of the theorem can be found in \cite[pp.~149--151]{Carmo}.

\begin{theorem}[Cartan-Hadamard] \label{CH}
Let $M^n$ be a complete and simply connected Riemannian manifold (of dimension $n$) with sectional curvature $K_x(\alpha)\le 0$ for all $x\in M$ and for all two-dimensional plane $\alpha\subset T_xM$. Then $M$ is diffeomorphic to $\mathbb{R}^n$, and the exponential map $\textup{exp}_x:T_xM\rightarrow M$ is diffeomorphism.
\end{theorem}
In words, the theorem implies the ``infiniteness'' of such manifold, in the sense that every geodesic (starting from any point and going in any direction) can be extended infinitely. 

\subsection*{Cheeger constant}
In \cite{Cheeger}, J. Cheeger introduced a constant $h$ of a manifold, representing an ``isoperimetric ratio'', and then proved an inequality that related this constant $h$ to $\lambda_1$, the smallest nonzero eigenvalue of Laplacian (see Section \ref{eigenlap}). The constant and the inequality were named after him as the Cheeger constant and Cheeger's inequality. 

\begin{definition}[The Cheeger constant]
The \textit{Cheeger constant} of a compact manifold $(M^n,g)$ is defined to be
$$h(M):=\inf\limits_{H} \frac{\textup{vol}_{n-1}(\partial H)}{\textup{vol}_n(H)}$$
where the infimum is taken over all compact submanifolds $H\subset M$ (of the same dimension) such that $0<\textup{vol}_n(H)\le \frac{1}{2}\textup{vol}_n(M)$, and $\partial H$ denotes the smooth boundary of $H$.

Moreover, in case $M$ is non-compact manifold, the \textit{Cheeger constant} (cf. Chavel, \cite[pp.~95]{Chavel}) is defined almost in the same way, except that the condition $\textup{vol}_n(H)\le \frac{1}{2}\textup{vol}_n(M)$ is removed. 

\end{definition}

For an advance notice, the two following theorems and their proofs involve Laplacian operator (whose definition and details can be found in Section \ref{LBsec} and \ref{eigenlap}). Some of formulas are not explained in this paper, but will be referred to \cite{Chavel,GHL,KarpP,Kasue}. 

\begin{theorem}[Cheeger's Inequality] 
Let $(M^n,g)$ be a compact Riemannian manifold. Then $$\lambda_1 \ge \frac{h(M)^2}{4},$$
where $\lambda_1$ is the first nonzero eigenvalue of Laplacian on $M$.
\end{theorem}

The proof we provide here follows from Gallot-Hulin-Lafontaine's \cite[pp.~238--240]{GHL} which proves in case $M$ is compact. Alternatively, Chavel provides a similar proof in \cite[pp.~95]{Chavel}, and the key argument is the co-area formula (see \cite[pp.~239]{GHL}, or \cite[pp.~85]{Chavel}). 
\begin{proof}
Suppose that $M$ is compact. Let $f$ be an eigenfunction corresponding to $\lambda_1$: $\Delta f+\lambda_1 f = 0$, and partition $M$ into three sets:
\begin{align*}
M_+ &= \{x\in M: f(x)>0\} \\
M_0 \ &= \{x\in M: f(x)=0\}\ \\
M_- &= \{x\in M: f(x)<0\}.
\end{align*}
Assume that $0$ is a regular value of $f$, that is, the preimage $M_0=f^{-1}(0)$ is a $(n-1)$-dimensional submanifold of $M$ (otherwise we can work with a function $f+\epsilon$ for arbitrary small $\epsilon$). Further assume $\textup{vol}(M_+)\le \frac{1}{2} \textup{vol}(M)$ (otherwise we can work with a function $-f$).

Performing integration by parts (in other words, integrating the Product rule \ref{PR} and then applying Divergence theorem), for any vector field $X$, we have
$$\int_{M_+} \langle \textup{grad } f, X \rangle + \int_{M_+}f\textup{div } X= \int_{M_+} \textup{div}(fX)=\int_{\partial M_+} \langle fX,\mathbf{\hat{n}}\rangle d\textup{vol}(\partial M_+)=0$$

because $f$ vanishes on the boundary of $M_+$ (which is $M_0$). In particular, choose $X=\textup{grad } f$, then the above equation can be read as
$$\int_{M_+} |\textup{grad } f|^2  = - \int_{M_+} f \Delta f = \lambda_1\int_{M_+} f^2.$$
Apply Cauchy-Schwarz's inequality and use that  $f|\textup{grad } f|=\frac{1}{2}|\textup{grad } f^2|,$
$$\lambda_1= \frac{\int_{M_+} |\textup{grad } f|^2}{\int_{M_+} f^2} \stackrel{C.S.}{\ge} 
\frac{(\int_{M_+} f|\textup{grad } f|)^2}{(\int_{M_+} f^2)^2}
= \frac{1}{4}\frac{(\int_{M_+} |\textup{grad } f^2|)^2}{(\int_{M_+} f^2)^2}.$$
The rest is to prove that $\int_{M_+} |\textup{grad } f^2|\ge h(M)\int_{M_+} f^2$. 

The co-area formula applied to the positive function $f^2$ gives
\begin{align*}
\int_{M_+} f^2&=\int_0^{\infty} \textup{vol}_{n}f^{-1}([\sqrt{t},\infty))dt=\int_0^{\infty} \textup{vol}_{n}( H_t)dt\\
\int_{M_+} |\textup{grad } f^2|&=\int_0^{\infty} \textup{vol}_{n-1}f^{-1}(\sqrt{t})dt=\int_0^{\infty} \textup{vol}_{n-1}(\partial H_t)dt
\end{align*}

where $H_t:=f^{-1}([\sqrt{t},\infty))$ is a submanifold of $M$ (or an empty set), with a smooth (or empty) boundary $\partial H_t=f^{-1}(\sqrt{t})$ for almost every $t$ (as long as $\sqrt{t}$ is regular value of $f$).

Moreover, 
$\textup{vol}(H_t)\le \textup{vol}(H_0)=\textup{vol}(M_+)\le \frac{1}{2}\textup{vol}(M)$, so by the definition of the Cheeger constant: 
$$\textup{vol}_{n-1}(\partial H_t)\ge h(M)\textup{vol}_n(H_t)$$ holds for almost every $t\ge 0$. Integration over $t\in[0,\infty)$ finally yields
$$\int_{M_+} |\textup{grad } f^2|=\int_0^{\infty} \textup{vol}_{n}( H_t)dt\ge h(M)\int_0^{\infty} \textup{vol}_{n-1}(\partial H_t)dt=h(M)\int_{M_+} f^2.$$

%The co-area formula states that, for a smooth positive function $F$ on compact $M_+$: $$\int_{M_+} |\textup{grad } F|=\int_0^{\textup{sup }F} \textup{vol}_{n-1}\big(F^{-1}(t) \big)dt,$$
%and $$\int_{M_+} F=\int_0^{\textup{sup }F}  \textup{vol}_{n}F^{-1}\big([t,+\infty] \big) dt.$$
%
%Applying the co-area formula with $F=f^2$, we then obtain
%\begin{align*}
%\int_{M_+} |\textup{grad } f^2| &= \int_0^{\textup{sup }f} \textup{vol}_{n-1}f^{-1}(\sqrt t)dt\\
%&\ge \int_0^{\textup{sup }f} h(M) \textup{vol}_{n}f^{-1}\big([\sqrt t, \textup{sup }f] \big)\\
%\end{align*}

\end{proof}
Next theorem asserts that Cheeger constant is strictly positive for a manifold whose curvature is negative and bounded away from zero. The discrete analogue of this theorem can be found in Theorem \ref{comcheegernegthm}.

\begin{theorem} \label{cheegernegthm}
Suppose that a complete manifold $M$ has negative sectional curvature bounded above by $-K_0<0$ (hence $M$ is non-compact, by Cartan-Hadamard). Then 
\begin{equation} \label{cheegerneg}
h(M)\ge (\textup{dim }M-1)\sqrt{K_0}
\end{equation}
where $h(M)$ is the Cheeger constant defined for non-compact $M$.
\end{theorem}

\begin{proof}
Let $H\subset M$ be a compact submanifold with a smooth boundary $\partial H$, and let $x_0\in M$ be a point such that $d(x_0,H)>0$. Consider the distance function $d_{x_0}(x):=d(x_0,x)$. Since $d(x_0,H)>0$ and $M$ has no conjugate points, when restricted to $H$ the function $d_{x_0}$ is smooth. Thus it makes sense to consider a Laplacian $\Delta d_{x_0}$ on $H$. First, by the Divergence theorem,
\begin{align*}
\int\limits_H \Delta d_{x_0} &= \int\limits_H \textup{div}(\textup{grad } d_{x_0} )
=\int\limits_{\partial H} \langle \textup{grad }d_{x_0}, \mathbf{\hat{n}} \rangle d\textup{vol}(\partial H)
\le \textup{vol}(\partial H) %\int\limits_{\partial H} d\textup{vol}(\partial H) = \textup{vol}(\partial H)
\end{align*}
where the above inequality is due to $\|\textup{grad }d_{x_0}\|\le 1$. In order to achieve \eqref{cheegerneg}, it suffices to show that $\displaystyle \Delta d_{x_0} \ge (\textup{dim }M-1)\sqrt{K_0}$.

In polar coordinates $(r,\phi)$, the Laplacian of a function $f=f(r,\phi)$ can be written as 
\begin{equation} \label{polarlap}
\Delta f(r,\phi) = \frac{\partial^2f}{\partial r^2}(r,\phi)+ H(r,\phi)\frac{\partial f}{\partial r}(r,\phi)+ \Delta^{S_r(x_0)}f(r,\phi)
\end{equation}
where $H(r,\phi)$ is the mean curvature, and $\Delta^{S_r(x_0)}$ is Laplacian restricted to $S_r(x_0)$ the sphere of radius $r$ centered at $x_0$. The derivation of this formula is analogous to the one given in \cite[Equation (2)]{KarpP}.

In particular, choose $f=d_{x_0}$, so it follows that $f(r,\phi)=r$ and $\frac{\partial f}{\partial r}=1$ and $\frac{\partial^2 f}{\partial r^2}=0$. Substitution into the equation \eqref{polarlap} then gives
$$\Delta d_{x_0} = H(r,\phi)+ \Delta^{S_r(x_0)}r=H(r,\phi),$$
because $\Delta^{S_r(x_0)}r=0$. By Hessian comparison theorem (see \cite[Lemma 2.18]{Kasue}), the condition on sectional curvature $K\le -K_0$ then implies that $$\Delta d_{x_0}=H(r,\phi)\ge \sqrt{K_0}\coth(-r\sqrt{K_0})(\textup{dim }M-1) \ge \sqrt{K_0}(\textup{dim }M-1)$$ as desired.

\end{proof}
\section{Laplacian operator and Bochner's formula} \label{LBsec}

In this section, we start with the definitions and properties of operators on Riemmanian manifolds, namely \textup{gradient}, \textup{divergence}, \textup{Laplacian}, and \textup{Hessian}. Then we state (without proof) Bochner's formula, which serves to be an essential background for Bakry-\'{E}mery curvature in Chapter \ref{BECchap}.

%Let $(M^n, g)$ be a Riemannian manifold. Moreover, let $f\in C^\infty(M)$ be a smooth real function on $M$, and $X\in\mathfrak X(M)$ be a smooth vector field on $M$.

\begin{definition}[Gradient, divergence and Laplacian] \ \\
\textit{Gradient operator} $\textup{grad}: C^\infty(M) \rightarrow \mathfrak X(M)$ maps a smooth real function $f$ to a smooth vector field \underline{$\grad f$} such that its evaluation at any point $x\in M$ is defined by the inner product: 
$$g_x(\grad f(x), w):= w(f)(x)$$
for every $w\in T_xM$. Here $w(f)$ is a differentiation of $f$ in direction of the vector $w$. 

\textit{Divergence operator} $\textup{div}: \mathfrak X(M) \rightarrow C^\infty(M)$ maps a smooth vector field $X$ to a smooth real function \underline{$\dive X$} defined at each point $x\in M$ by
$$(\dive X)(x) := \mbox{tr}_{T_x M} (v \mapsto \nabla_v X)$$
where the mapping is considered from the tangent space $T_xM$ onto itself, and $\nabla$ is Levi-Civita connection.

\textit{Laplacian operator} $\Delta: C^\infty(M) \rightarrow C^\infty(M)$ is then defined to be the composition: $ \Delta = \mbox{div} \circ \mbox{grad}$. 
\end{definition}

\begin{proposition} \label{localgrad}
In local coordinates,
\begin{equation} \label{grad}
\grad f(x)=\sum\limits_i \bigg(\sum\limits_j \diff{j}(x) g^{ij}(x)\bigg)\dif{i}
\end{equation} 
and 
\begin{equation} \label{div}
(\dive X)(x)=\frac{1}{\sqrt{\mbox{det }g(x)}}\sum\limits_i \dif{i}\big(\sqrt{\mbox{det }g(x)}\cdot X_i(x)\big)
\end{equation} 
where $g(x)$ is an $n\times n$ matrix whose $ij$-entry is $\displaystyle g_{ij}(x):=g_x(\dif{i},\dif{j})$, and $g^{ij}(x)$ is the $ij-entry$ of the inverse matrix $g^{-1}(x)$. Moreover, $X_i$ is the $i$-coordinate of vector field $X$, that is $\displaystyle X(x)=\sum\limits_i X_i(x)\dif{i}$ written in local coordinates.
\end{proposition}

Since most of the time, functions are evaluated at a fixed point $x$, without ambiguity we may omit the terms $x$ in the writing. Moreover, we write $\displaystyle \partial_i:=\dif{i}$ and $\langle \cdot,\cdot\rangle := g_x(\cdot,\cdot)$.

\begin{proof}
Write vector field $\displaystyle \grad f(x)= \sum\limits_i a_i(x)\dif{i}$ with respect to local coordinates (or in short, $\grad f= \sum\limits_i a_i\partial_i$).

\noindent By definition of gradient, we have $$\diff{j} = \langle\grad f, \partial_j\rangle = \sum\limits_i a_{i} \langle\partial_i, \partial_j\rangle = \sum\limits_i a_{i} g_{ij}.$$
It follows that, for a fixed $k$,
\begin{align*}
\sum\limits_j \diff{j} \cdot g^{jk}
&=\sum\limits_{i,j} a_{i} g_{ij} g^{jk}\\
&=\sum\limits_i a_{i} \bigg(\sum\limits_j g_{ij}g^{jk}\bigg)\\
&=\sum\limits_i a_{i}\delta_{ik} = a_k
\end{align*}
where $\delta_{ik}$ is Kronecker delta. The equation \eqref{grad} immediately follows.

In the definition of $\dive X$, the mapping 
$v\mapsto \nabla_v X$ can be represented as a matrix $B=[b_{ij}]$ with respect to a (basis) frame $\{E_i\}_{i=1}^n$ of $TM$. Then $$\dive X=\textup{tr} (B) = \sum\limits_{i=1}^n b_{ii} = \sum\limits_{i=1}^n \langle \nabla_{E_i}X, E_i\rangle,$$ which is independent to the choice of frame $E_i$'s (not needed to be orthonormal).

In particular, choose $E_i=\partial_i$ for all $i$, we have
\begin{align*}
 \nabla_{\partial_i}X
&=\sum\limits_j \nabla_{\partial_i}( X_j\partial_j)\\
&= \sum\limits_j \bigg(\frac{\partial}{\partial x_i}X_j\cdot\partial_j +X_j\cdot\nabla_{\partial_i}\partial_j\bigg)\\
&= \sum\limits_j \bigg(\frac{\partial}{\partial x_i}X_j\cdot\partial_j +X_j\sum\limits_{k}\Gamma_{ij}^k\partial_k\bigg)
\end{align*}
where $\Gamma_{ij}^k$'s are the Christoffel symbols.\\

\noindent Therefore,
\begin{equation} \label{divtr}
\dive X=\sum\limits_{i=1}^n \langle \nabla_{\partial_i}X, \partial_i\rangle=\sum\limits_i\bigg(\frac{\partial}{\partial x_i}X_i+X_i\sum\limits_k \Gamma_{ki}^k \bigg)
\end{equation}

\noindent where
\begin{align} \label{cs}
\sum\limits_k \Gamma_{ki}^k &= \frac{1}{2} \sum\limits_k \sum\limits_l g^{kl}(\frac{\partial}{\partial x_i}g_{kl}+\frac{\partial}{\partial x_k}g_{il}-\frac{\partial}{\partial x_l}g_{ki}) \nonumber\\
&=\frac{1}{2} \sum\limits_{k,l}g^{kl}\frac{\partial}{\partial x_i}g_{kl}
\end{align}

\noindent On the other hand, for each fixed $i$, we have 
\begin{align} \label{detg}
 \frac{\partial}{\partial x_i} (\sqrt{\mbox{det }g}\cdot X_i)
&=\sqrt{\mbox{det }g} \frac{\partial}{\partial x_i}X_i+X_i\frac{\partial}{\partial x_i}\sqrt{\mbox{det }g}
\end{align}

\noindent where the derivative term $\displaystyle \frac{\partial}{\partial x_i}\sqrt{\textup{det }g}$ can be calculated as
\begin{align*}
\frac{\partial}{\partial x_i}\sqrt{\mbox{det }g} &= \frac{1}{2\sqrt{\mbox{det }g}} \frac{\partial}{\partial x_i} \mbox{det }g\\
&= \frac{1}{2}\sqrt{\mbox{det }g} \cdot \mbox{tr}\bigg(g^{-1}\cdot \frac{\partial g}{\partial x_i}\bigg) &&\textup{(Jacobi's formula)}\\
&= \frac{1}{2}\sqrt{\mbox{det }g}\cdot  \sum\limits_{k,l} g^{kl}\frac{\partial }{\partial x_i}g_{lk}\\
&\stackrel{\eqref{cs}}{=} \sqrt{\mbox{det }g} \cdot\sum\limits_k \Gamma_{ki}^k
\end{align*}

\noindent Summing equation \eqref{detg} over index $i$, we obtain the desired equation \eqref{div}:
$$\sum\limits_i \frac{\partial}{\partial x_i} (\sqrt{\mbox{det }g}\cdot X_i) \stackrel{\eqref{divtr}}{=} \sqrt{\mbox{det }g}\cdot\dive X.$$

\end{proof}

The explicit calculation of $\Delta$ in local coordinates follows immediately from Proposition \ref{localgrad}.

\begin{corollary}
In local coordinates, $$\Delta f=\frac{1}{\sqrt{\dete g}} \sum\limits_i \frac{\partial}{\partial x_i}\bigg(\sqrt{\dete g} \sum\limits_j \diff{j}g^{ji}\bigg).$$
\end{corollary}

\begin{remark}
For example, in $\mathbb{R}^n$ with Euclidean metric, the laplacian $\Delta$ is 
$$\Delta f= \sum\limits_{i=1}^n \frac{\partial^2 f}{\partial x_i^2}$$
\end{remark}

The following proposition is the product rule of gradient, divergence, and Laplacian.
\begin{proposition}[Product rule] \label{PR}
Let $f,h\in C^\infty(M)$ be smooth functions and $X\in\mathfrak{X}(M)$ be a vector field. Then
\begin{itemize}
\item[\textup{(a)}] $\textup{grad}(fh)= f\grad h+h\grad f$
\item[\textup{(b)}] $\textup{div}(fX) = \langle \grad f, X\rangle + f\dive X$
\item[\textup{(c)}] $\Delta(fh)= 2\langle \grad f, \grad h \rangle +f\Delta h+h\Delta f$
\end{itemize}
\end{proposition}

\begin{proof}
By the definition of gradient and divergence, the product rule in part (a) and (b) is induced from the product rule of directional derivative and the product rule of Levi-Civita connection, respectively.
For part (c), we have
\begin{align*}
\Delta(fh)&=\textup{div}(\textup{grad}(fh))\\
&=\textup{div}(f\grad h)+\dive(h\grad f)\\
&=\langle \grad f, \grad h\rangle +f\textup{div}(\grad h) + \langle \grad h, \grad f\rangle +h\textup{div}(\grad f)\\
&=2\langle \grad f, \grad h\rangle+f\Delta h+h\Delta f.
\end{align*}

\end{proof}

\begin{definition} [Hessian]
For a smooth function $f\in C^\infty(M)$, the Hessian tensor, $\textup{Hess}(f)$, is a bilinear form defined as $$\textup{Hess}(f) (X,Y) := \langle \nabla_X \grad f, Y \rangle$$ for any $X,Y\in TM$.
\end{definition}

A fundamental property of the Hessian is symmetry: 
\begin{proposition}
$\textup{Hess}(f) (X,Y)= \textup{Hess}(f) (Y,X)$
\end{proposition}

\begin{proof} We have
\begin{align*}
\textup{Hess}(f) (X,Y) &= \langle \nabla_X \grad f, Y \rangle\\
&= X\langle \grad f, Y\rangle - \langle \grad f, \nabla_X Y\rangle \\
&= X(Yf)-\langle \grad f, \nabla_X Y\rangle 
\end{align*}
\noindent where in the second line of equations, we use the metric property of $\nabla$: $X\langle Y,Z \rangle=\langle \nabla_X Y,Z\rangle+ \langle Y,\nabla_X Z \rangle$. 

\noindent Similarly, we have $$ \textup{Hess}(f) (Y,X)=Y(Xf)-\langle \grad f, \nabla_Y X\rangle,$$ and therefore
\begin{align*}
\textup{Hess}(f) (X,Y) - \textup{Hess}(f) (Y,X) &= X(Yf)-Y(Xf) - \langle \grad f, \nabla_X Y - \nabla_Y X\rangle\\
&=[X,Y](f) - \langle \grad f, [X,Y] \rangle \\
&= [X,Y](f)- [X,Y](f) \\ &= 0
\end{align*}
\noindent where $[X,Y]$ is the Lie bracket of vector fields, and the second line of equations is due to the torsion-freeness of $\nabla$: $\nabla_X Y - \nabla_Y X = [X,Y]$. \end{proof}

The Hessian tensor can also be represented by a matrix $A=[a_{ij}]$ w.r.t. an arbitrary orthonormal frame $\{E_i\}_{i=1}^n$ of $TM$, that is
$$a_{ij}= \textup{Hess}(f) (E_i, E_j).$$

\noindent Moreover, the norm $||\textup{Hess }f ||$ is defined as in Hilbert-Schmidt norm:
$$|| \textup{Hess }f|| := \sqrt{\textup{tr} (AA^t)}=\sqrt{ \sum\limits_{i,j} a_{ij}^2}$$
\noindent which is independent to the choice of orthonormal frame $E_i$'s.

\begin{proposition} \label{HL}
The following two relations hold between Hessian and Laplacian.
\begin{itemize}
\item[\textup{(a)}] $\textup{tr}(\textup{Hess } f) = \Delta f$
\item[\textup{(b)}] $|| \textup{Hess }f||^2 \ge \frac{1}{n}(\Delta f)^2$
\end{itemize}
\end{proposition}
\begin{proof}
Part (a) follows directly from definitions:
$$\textup{tr}(\textup{Hess } f)=\sum\limits_{i=1}^n a_{ii}=\sum\limits_{i=1}^n \langle \nabla_{E_i} \grad f, E_i \rangle =\dive(\grad f) =\Delta f.$$

\noindent For part (b), we apply Cauchy-Schwarz's inequality to part (a):
\begin{align*}
|| \textup{Hess }f||^2 = \sum\limits_{i,j} a_{ij}^2  \ge \sum\limits_{i} a_{ii}^2 \stackrel{\small{C.S.}}{\ge} \frac{1}{n}\bigg(\sum\limits_{i} a_{ii}\bigg)^2 =\frac{1}{n}(\Delta f)^2.
\end{align*}
\end{proof}

We are now ready for the statement of Bochner's formula, an equation that merges the defined operators together and connects to Ricci curvature. This formula is a fundamental motivation of the Bakry-\'{E}mery curvature notion introduced in Chapter 3. We omit the proof of the formula; see \cite[Proposition 4.15]{GHL} for details.
\begin{theorem}[Bochner's formula] \label{bochner}
Let $(M^n,g)$ be a Riemannian manifold. For any smooth function $f\in C^\infty(M)$, the identity
\begin{equation} \label{bochner}
\frac{1}{2}\Delta |\grad f|^2 = ||\textup{Hess }f||^2 + \langle \grad\Delta f, \grad f\rangle +\textup{Ric}(\grad f),
\end{equation}
\noindent holds pointwise on $M$.
\end{theorem}

\section{Eigenvalues of Laplacian and Lichnerowicz theorem} \label{eigenlap}

Let $(M^n,g)$ be a compact connected Riemannian manifold. \textit{Eigenvalues} of Laplacian operator on $M$ are real numbers $\lambda$ such that there exists a nontrivial solution $f\in C^2(M)$ (i.e. twice continuously differentiable) to the system of equations
\begin{align*}
\Delta f + \lambda f=0 &\mbox{ on } M\\
f=0 &\mbox{ on } \partial M.
\end{align*}
Such function $f$ is called an \textit{eigenfunction} corresponding to $\lambda$. In case, $M$ is a closed manifold, the condition $f=0$ on $\partial M$ may be removed.

The eigenvalues $\lambda$ of Laplacian are known to be real, positive, and arrangeable in an increasing order (see \cite[pp.~8]{Chavel}): $$0=\lambda_0 < \lambda_1 \le \lambda_2 \le ... , \mbox{where  } \lambda_n\rightarrow \infty \mbox{ as } n\rightarrow\infty.$$

In Lichnerowicz's theorem, the first (i.e. smallest) nonzero eigenvalue $\lambda_1$ is estimated from below, under an assumption that Ricci curvature is strictly positive and and bounded away from zero. Here we prove in a special case where $M$ is a closed manifold. A proof in general case where $M$ is compact can be referred to e.g. \cite[Theorem 4.70]{GHL}.
\begin{theorem}[Lichnerowicz]
Let $(M^n,g)$ be a closed Riemannian manifold. Suppose that $\textup{Ric}_x(v)\ge K>0$ for all $x\in M$ and $v\in S_xM$ (that is, $v\in T_xM$ and $|v|=1$). Then $\lambda_1 \ge \frac{n}{n-1}K$.
\end{theorem}

\begin{proof}
Consider an eigenfunction $f$ satisfying $\Delta f +\lambda f= 0$. Upon scalar multiplication of $f$, we may further assume that $\int\limits_M f^2=1$. Bochner's formula (see Theorem \ref{bochner}) gives
\begin{align*}
\frac{1}{2}\Delta |\grad f|^2 &= ||\mbox{Hess }f||^2 + \langle\grad (-\lambda f), \grad f\rangle +\mbox{Ric}(\grad f)\\
&=||\mbox{Hess }f||^2 - \lambda |\grad f|^2+\textup{Ric}(\grad f)
\end{align*} 
The curvature assumption can also be expressed as $\textup{Ric}_x(v) \ge K|v|^2$ for all $v\in T_xM$.
By applying Proposition \ref{HL} and this curvature assumption to Bochner's formula above, we obtain
\begin{align*}
\frac{1}{2}\Delta |\grad f|^2 &\ge  \frac{1}{n}(\Delta f)^2 -\lambda |\grad f|^2+K|\grad f|^2\\
&=\frac{\lambda^2}{n}f^2 + (K-\lambda)|\grad f|^2\\
&=\frac{\lambda^2}{n}f^2 + (K-\lambda)(\frac{1}{2}\Delta f^2+\lambda f^2)
\end{align*} 
\noindent where the last equality is the product rule : $\Delta f^2 =2|\grad f|^2+2f\Delta f$.\\\\
Integrate the above inequality over $M$, and use the fact that $\int\limits_M \Delta |\grad f|^2=0$ and $\int\limits_M \Delta f^2=0$ by the Divergence Theorem (since $M$ has no boundary). We then have $$0\ge \frac{\lambda^2}{n} +(K-\lambda)\lambda.$$

\noindent In particular, for $\lambda_1>0$, we have $\lambda_1 \ge  \frac{n}{n-1}K$ as desired.
\end{proof}

\section{Bonnet-Myers theorem} \label{BMsec}

Bonnet-Myers theorem is a classical theorem in Riemannian geometry. It states that a connected and complete manifold with Ricci curvature bounded below by a positive number must be compact. It is the main theorem of our paper that we will discuss about in all of the discrete curvature notions in later chapters.
\begin{theorem}[Bonnet-Myers] \label{BM}
Let $(M^n,g)$ be a connected and complete Riemannian manifold. Suppose there is a constant $r>0$ such that the Ricci curvature satisfies $$Ric_x(v)\ge \frac{n-1}{r^2}>0$$ for all $x\in M$ and $v\in S_x(M)$. Then $M$ is compact and its diameter $\textup{diam}(M)\le \pi r$.
\end{theorem}

One way to prove this theorem is to apply the second variation formula of the energy, as presented by Carmo \cite[pp.~191--201]{Carmo}).

\begin{proof}
Let $p,q\in M$ be two arbitrary points in $M$. By Hopf-Rinow theorem (see \cite[pp.~145--148]{Carmo}), there exists a minimal unit-speed geodesic $c:[0,a]\rightarrow M$ joining $x$ and $y$, that is $c(0)=p$, $c(a)=q$, $|c'(t)|=1$ for all $t$, and $d(p,q)=\ell(c)=a$. It suffices to prove that $a\le \pi r$, because we can then conclude $\textup{diam}(M)\le \pi r$ and the compactness of $M$ (from its being complete and bounded).

First, we will construct proper variations of $c$ as follows. Choose unit vectors $e_1,e_2,...,e_{n-1}$ in $T_p(M)$ such that they, together with $c'(0)$, form an orthonormal basis of $T_pM$. For each $1\le i \le n-1$, let $V_i$ be a parallel vector field along $c$ such that $V_i(0)=e_i$. Note that $$\frac{d}{dt} \langle V_i(t), V_j(t) \rangle = \langle \frac{D}{dt} V_i(t), V_j(t)\rangle + \langle V_i(t), \frac{D}{dt}  V_j(t)\rangle =0$$
\noindent since $\frac{D}{dt}  V_i(t)= \frac{D}{dt}  V_j(t)=0$ from being parallel. It means $\langle V_i(t), V_j(t) \rangle$ is constant and $\langle V_i(t), V_j(t) \rangle= \langle e_i, e_j \rangle = \delta_{ij}$. 

For each $i$, define $X_i(t):=\sin(\frac{\pi t}{a})V_i(t)$, and let $F_i:(-\varepsilon,\varepsilon)\times [0,a]\rightarrow M$ be a variation of $c$ whose variational vector field is $X_i$, that is 
$$F_i(0,t)=c(t) \ \ \ \ \ \mbox{ and } \ \ \ \ \ \frac{\partial}{\partial s} F_i(s,t)=X_i(t).$$
Since $X_i(0)=X_i(a)=0$, it means that for every $s\in(-\varepsilon,\varepsilon)$ the curve $F_i(s, -)$ has the same endpoints as the curve $c$. In other words, $F_i$ is a proper variation of $c$.

\noindent The energy for the curve $F_i(s, -)$ is defined by $$E_i(s):=\frac{1}{2}\int\limits_0^a \bigg\|\frac{d}{dt}F_i(s,t)\bigg\|^2 dt.$$
The second variation formula of energy states that 
\begin{align} \label{SVF}
E_i''(0)&=\int\limits_0^a\Big|\frac{D}{dt}X_i(t)\Big|^2-\Big\langle X_i(t), R\Big(c'(t),X_i(t)\Big)c'(t)\Big\rangle dt \nonumber \\
&=\int\limits_0^a\Big|\frac{\pi}{a}\cos(\frac{\pi t}{a})V_i(t)\Big|^2-\sin^2(\frac{\pi t}{a})\Big\langle V_i(t), R\Big(c'(t),V_i(t)\Big)c'(t)\Big\rangle dt \nonumber \\
&=\int\limits_0^a \frac{\pi^2}{a^2}\cos^2(\frac{\pi t}{a})-\sin^2(\frac{\pi t}{a}) K(c'(t),V_i(t)) dt 
\end{align}
\noindent where $K(c'(t),V_i(t))$ is sectional curvature of the two-dimensional plane spanned by $c'(t)$ and $V_i(t)\}$.\\\\
Summing the equation \eqref{SVF} over index $i$ and using the fact that $$\sum\limits_{i=1}^{n-1}K(c'(t),V_i(t))=Ric(c'(t))\ge \frac{n-1}{r^2},$$ we then have
\begin{align} \label{sume}
\sum\limits_{i=1}^{n-1} E_i''(0) &\le (n-1)\int\limits_0^a \frac{\pi^2}{a^2}\cos^2(\frac{\pi t}{a})-\frac{1}{r^2}\sin^2(\frac{\pi t}{a}) dt \nonumber \\
&=(n-1)\bigg(\frac{\pi^2}{a^2}-\frac{1}{r^2}\bigg)\frac{a}{2}
\end{align}

Since $c$ is minimal geodesic (with constant speed), its energy $E_i(0)$ is minimum among $E_i(s)$, $s\in(-\varepsilon,\varepsilon)$. Hence $E_i''(0)\ge 0$, true for every $i$. The relation \eqref{sume} then implies $a\le \pi r$ as desired.
\end{proof}

\begin{remark} The diameter bound $\textup{diam}(M)\le \pi r$ is sharp for the round sphere $S^n_r:=\{x\in \mathbb{R}^{n+1}: \|x\|=r\}$. More importantly, the only manifolds for which the bound is sharp are the ones that are isometric to the round sphere $S^n_r$; this result is known as \textit{Cheng's rigidity result} (see \cite{Cheng}).
\end{remark}

\section{Average distance between two balls} \label{avgsec}
 In \cite{Ollivier}, Ollivier suggests that for two points $x$ and $y$ of a manifold, the average distance between small balls centered at $x$ and at $y$ can be greater or smaller than the distance $d(x,y)$ depending on the Ricci curvature. This statement can be explained more precisely as follows.

Let $(M^n,g)$ be a connected and complete manifold. Let $x$ and $y$ be two points in $M$. By Hopf-Rinow theorem, completeness of $M$ implies that there exists a minimal geodesic $c$ joining $x$ and $y$. Further assume that $c$ is unit speed, so $c$ can be parametrized as $c:[0,\delta]\rightarrow M$ with $c(0)=x$, $c(\delta)=q$, $|c'(t)|=1$ for all $t$, and hence $d(x,y)=\ell(c)=\delta$. Let $v=c'(0)\in S_xM$ be the initial unit velocity of curve $c$.  Define $B_r(x)=\{z\in M:\ d(x,z)\le r\}$ to be the ball of a small radius $r$ around $x$, and define $B_r(y)$ similarly. 

%We are interested in finding the average distance between corresponding points in $B_r(x)$ and $B_r(y)$. One might guess that the answer is $\delta$, which is true if the manifold is flat (i.e. having zero curvature everywhere). However, a more correct answer is given in terms of curvatures, explained as follows.

For each point $x'\in B_r(x)$, let $d(x',x)=:\varepsilon \le r$. We can write $x'=\textup{exp}_x(\varepsilon w)=c_w(\epsilon)$, which means $x'$ can be reached by the unit-speed geodesic $c_w$, starting from $x$ with the initial unit velocity $w \in S_xM$ and travelling for a period of time $\varepsilon$. Consider $w':=P_c^\delta(w)$, the parallel transport of $w$ along the curve $c$ for a period of time $\delta$. Thus $w'\in S_yM$.
Then $y'\in B_r(y)$ is a corresponding point of $x'\in B_r(x)$, given by $y':= \textup{exp}_y(\varepsilon w').$ See Figure \ref{fig:dxy}. The first task is to estimate the distance $d(x',y')$, and the second task is to derive the average distance of $\overline{d}(B_r(x),B_r(y))$ by averaging over all $x'\in B_r(x)$. 

\begin{figure}[h] 
\centering
\includegraphics[width= 12cm]{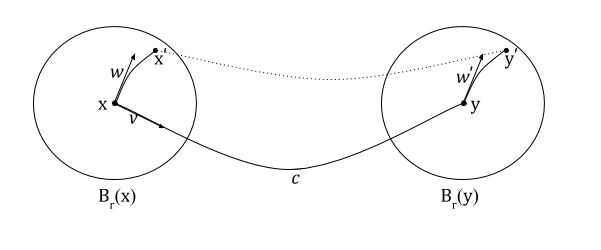}
\caption{Correspondence between $x'$ and $y'$}
\label{fig:dxy}
\end{figure}

\begin{proposition}
In the above setting (with further assumption that $w\perp v$), the distance between $x'$ and $y'$ is estimated by
\begin{align*}
d(x',y')\le \delta\Big(1-\varepsilon^2K(v,w)+O(\varepsilon^3+\varepsilon^2\delta)\Big)
\end{align*}
\noindent where $K(v,w)$ is the sectional curvature of the two-dimensional plane spanned by $\{v,w\}$.
\end{proposition}

\begin{proof}
For $s\in[0,\delta]$, let $v_s:=\frac{d}{ds}c(s)$ be the velocity of the curve $c$ at time $s$, and let $w_s:=P_c^s(w)$ be the parallel transport of $w$ along the curve $c$ for a period of time $s$. Therefore, $v_s,w_s\in S_{c(s)}M$ for all $s\in [0,\delta]$, and $\langle v_s, w_s \rangle$ is constant in $s\in[0,\delta]$. Moreover, $v_0=v$, $w_0=w$, and $w_\delta=w'$.

Consider $F:[0,\delta]\times[0,\varepsilon]\rightarrow M$ a geodesic variation defined by $$F(s,t)=c_s(t):= \textup{exp}_{c(s)}(tw_s).$$
\noindent i.e. $F(s,-)=c_s$ is a geodesic for every $s\in[0,\delta]$. For a fixed $s_0$, let $J_{s_0}$ be a variational vector field associated to the variation $F$ of the geodesic $c_{s_0}$, that is $$\frac{\partial}{\partial s}F(s_0,t)=J_{s_0}(t).$$
Hence, $J_{s_0}$ is a Jacobi field along $c_{s_0}$ and satisfies the Jacobi equation $$J_{s_0}''(t)+R(c'_{s_0}(t), J_{s_0}(t))c'_{s_0}(t)=0$$ \noindent where $J_{s_0}'(t)=\frac{D}{dt}J_{s_0}(t)=\nabla_{c'_{s_0}(t)} J_{s_0}$ and $J_{s_0}''(t)=\frac{D^2}{dt^2}J_{s_0}(t)=\nabla_{c'_{s_0}(t)}(\nabla_{c'_{s_0}(t)} J_{s_0})$.\\

Let $\gamma:[0,\delta]\rightarrow M$ be the curve $\gamma(s):=c_s(\varepsilon)$ from $\gamma(0)=x'$ to $\gamma(\delta)=y'$. We aim to compute the length of $\gamma$ which then becomes an upper bound for the distance $d(x',y')$. First, note that $$\gamma'(s_0)=\frac{d}{ds}\bigg|_{s_0} \gamma(s)=\frac{d}{ds}\bigg|_{s_0} F(s,\varepsilon)=J_{s_0}(\varepsilon).$$
Choose $f(t)=\|J_{s_0}(t)\|^2$, then it follows that $$\|\gamma'(s_0)\|^2=f(\varepsilon)=f(0)+\varepsilon f'(0)+\frac{\varepsilon^2}{2}f''(0)+O(\varepsilon^3)$$
\noindent The terms $f(0)$, $f'(0)$, and $f''(0)$ can be calculated as follows:
\begin{itemize}
\item $\displaystyle f(0)=\|J_{s_0}(0)\|^2=\|c'(s_0)\|^2=1.$

\item $\displaystyle f'(0)=\frac{d}{dt}\bigg|_{t=0} \langle J_{s_0}(t), J_{s_0}(t) \rangle =2\langle J'_{s_0}(0), J_{s_0}(t)\rangle =0$ because
\begin{align*}
J'_{s_0}(0)= \frac{D}{dt}\bigg|_{t=0} J_{s_0}(t) &= \frac{D}{dt}\bigg|_{t=0} \frac{d}{ds}\bigg|_{s_0} F(s,t)\\
&= \frac{D}{ds}\bigg|_{s_0} \frac{d}{dt}\bigg|_{t=0} F(s,t) &&\text{(symmetry lemma)}\\
&= \frac{D}{ds}\bigg|_{s_0} w_s=0 &&\text{($w_s$ is parallel along $c$)}
\end{align*}

\item $ \displaystyle
\begin{aligned}[t]
\frac{1}{2}f''(0)&=\|J'_{s_0}(0)\|^2 +\langle J_{s_0}(0), J''_{s_0}(0)\rangle \\
&= 0 - \langle J_{s_0}(0), R(c'_{s_0}(0), J_{s_0}(0))c'_{s_0}(0) \rangle  &&\text{(Jacobi equation)}\\
&= - \langle v_{s_0}, R(w_{s_0}, v_{s_0}) w_{s_0}\rangle\\
&= - \langle R(w,v)w, v\rangle +O(\delta)
\end{aligned}
$ \\ where the last equality holds true by a linear approximation of a continuously differentiable function $A(s):= \langle v_{s}, R(w_{s}, v_{s}) w_{s}\rangle$ around $s=0$.
\end{itemize}

\noindent Therefore,
\begin{align*}
\|\gamma'(s_0)\|^2&=f(\varepsilon)=f(0)+\varepsilon f'(0)+\frac{\varepsilon^2}{2}f''(0)+O(\varepsilon^3)\\
&=1-\varepsilon^2\langle R(w,v)w, v\rangle + O(\varepsilon^3+\varepsilon^2\delta)\\
&=1-\varepsilon^2K(v,w) + O(\varepsilon^3+\varepsilon^2\delta)
\end{align*}
under the assumption that $\langle v,w\rangle=0$.\\
Hence, the length of $\gamma$ is 
$$\int\limits_0^\delta \|\gamma'(s_0)\|^2 ds=\delta(1-\varepsilon^2\langle R(w,v)w, v\rangle + O(\varepsilon^3+\varepsilon^2\delta))$$ which yields the proposition.
\end{proof}

In fact, the inequality sign in the proposition can be replaced by the equality, as $\varepsilon \rightarrow 0$ and $\delta\rightarrow 0$ (see Proposition 6 in \cite{Ollivier}). Moreover, the averaging procedure as discussed in \cite[pp.~58]{Ollivier} yields the average distance between $B_r(x)$ and $B_r(y)$:

\begin{align} \label{avg_dist}
\overline d(B_r(x),B_r(y)) =
\delta\Big(1-\frac{r^2}{(N+2)}\textup{Ric}_x(v) + O(r^3+r^2\delta) \Big).
\end{align}

%$W_1(\mu^*_x,\mu^*_y)$ the minimal cost of transport plan from $\mu^*_x$, the uniform distribution in the ball $B_r(x)$ to $\mu^*_y$, the uniform distribution in the ball $B_r(y)$, as $r\rightarrow 0$ and $\delta\rightarrow 0$. 

\newpage
\section{Examples of graphs and manifolds} \label{expsec}
We introduce three examples of graphs that represent different classes of manifolds, and calculate curvature in Bakry-\'Emery and Ollivier's Ricci notations. Definition of graphs can be found at the beginning of Section 2.1, and details of these two curvature notations are provided in Chapter 3 and 4. We then verify that the nature of curvatures in such graphs correspond to the manifolds they represent. Here, the curvature calculation is performed in Graph Curvature Calculator written by Stagg and Cushing (see \cite{Graphcal} and the website \href{http://www.mas.ncl.ac.uk/graph-curvature/}{http://www.mas.ncl.ac.uk/graph-curvature/}), in the setting of ``normalized laplacian (with $\infty$ dimension)'' for Bakry-\'Emery curvature, and ``Lin-Lu-Yau'' for Ollivier's Ricci curvature. 
\begin{example}
The \textit{hypercube} graph $Q^n$ is the graph formed by vertices and edges of the $n$-dimensional hypercube. There are in total $2^n$ vertices of the form $x=(x_1,...,x_n)$ with each $x_i\in \{0,1\}$. An edge connects vertices $x$ and $y$ if and only if their coordinates differ by exactly one digit. The hypercube graph $Q^n$ represents the round sphere $S^{n-1}=\{x\in \mathbb{R}^{n}: \|x\|=1\}$. Similar to the round sphere, the hypercube graph has positive constant curvature everywhere, as shown in Figure \ref{fig:hypercube}
\begin{figure}[h]
\centering
\includegraphics[width= 14cm]{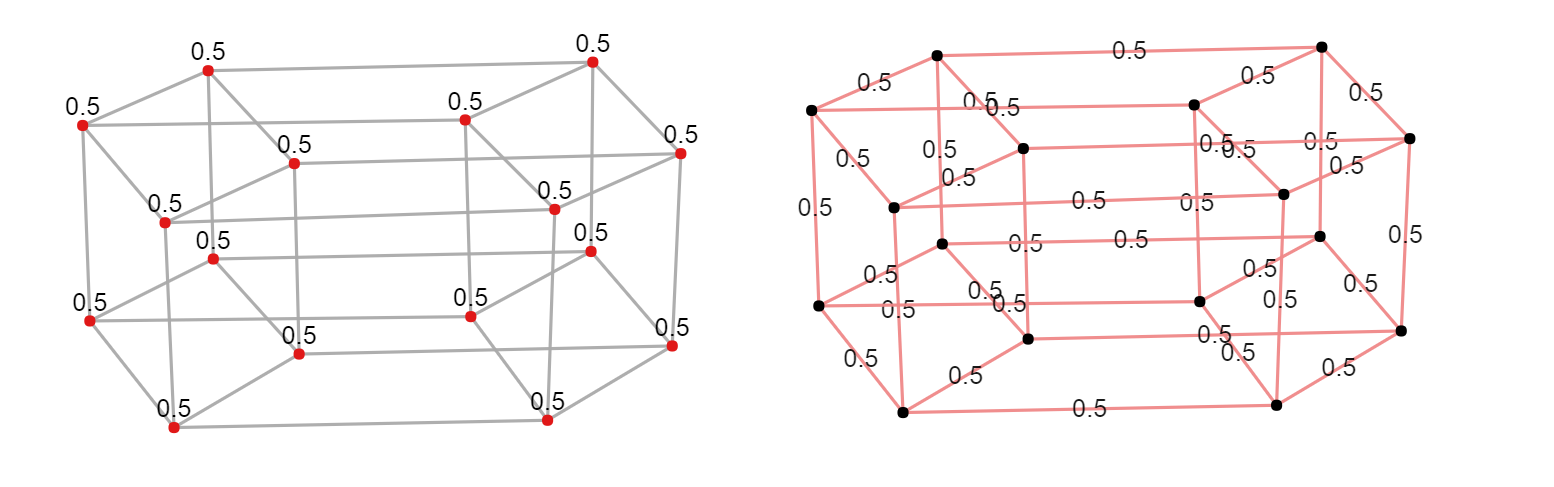}
\caption{Bakry-\'Emery and Ollivier's Ricci curvature of the hypercube graph $Q^4$}
\label{fig:hypercube}
\end{figure}
\end{example}

\begin{example}
The \textit{antitree} graph is the infinite graph constructed by placing complete graphs $K_n$, $n=1,2,3,...$ (in an increasing order of $n$) and connecting every vertex of $K_i$ to every vertex of $K_{i+1}$ for all $i\in \mathbb{N}$. The antitree graph represents a (elliptic) paraboloid. A paraboloid is a manifold with positive curvature everywhere, but its curvature is reaching zero at a point further away from the paraboloid's vertex. It is good to note that Bonnet-Myers theorem does not apply, and a paraboloid is indeed non-compact. As shown in Figure \ref{fig:antitree}, the curvature of the antitree is calculated to be \{0.5, 0.212, 0.092, 0.049, ...\} in Bakry-\'Emery curvature and \{0.6, 0.15, 0.068, 0.039, ...\} in Ollivier's Ricci curvature. This calculation suggests evidentially that the antitree is an infinite graph whose curvature is also reaching zero. A precise formula to calculate curvature of a generalized family of antitrees can be found in \cite{CCantitree}.

\begin{figure}[h] 
\centering
\includegraphics[width= 12cm]{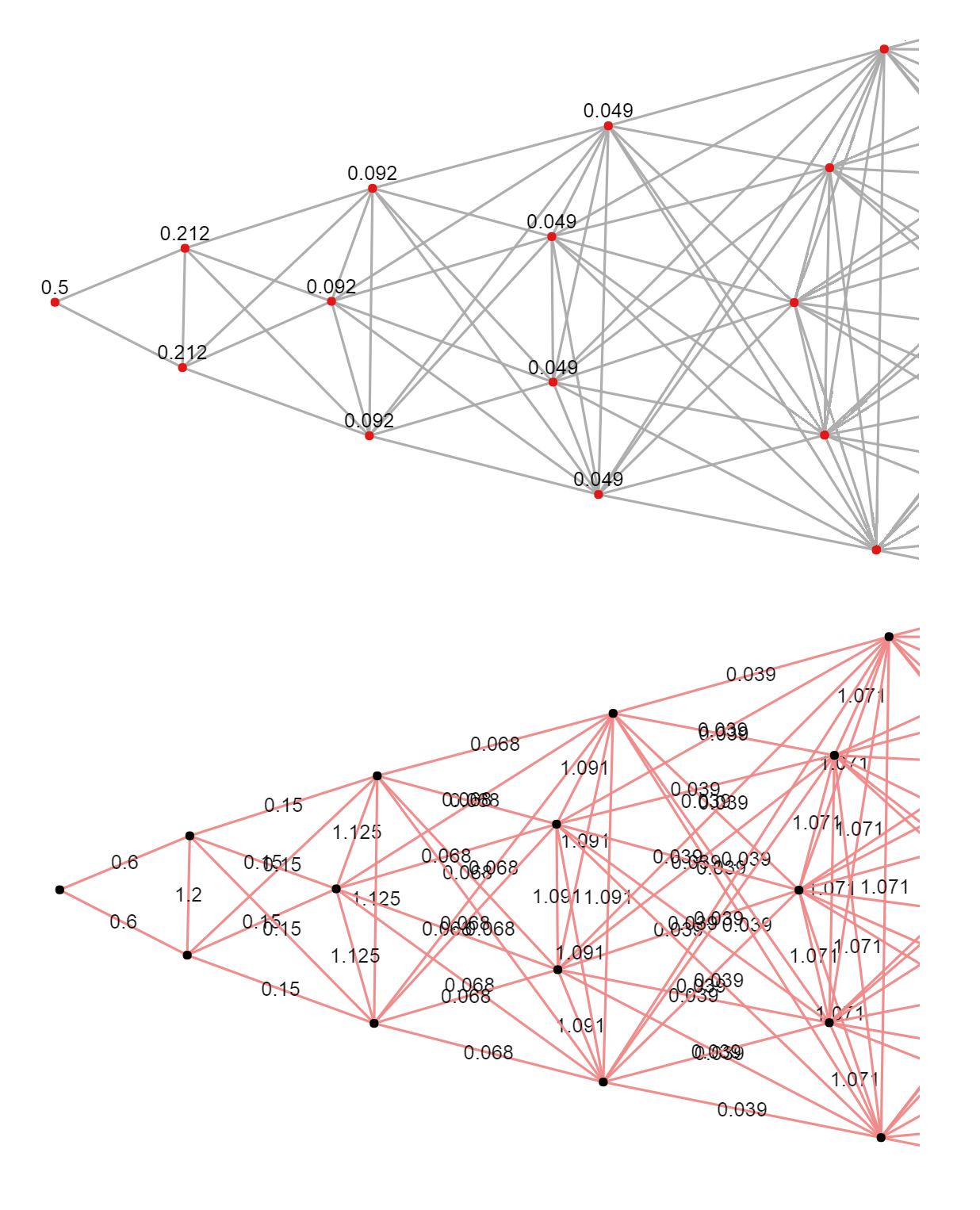}
\caption{Bakry-\'Emery and Ollivier's Ricci curvature of the antitree graph}
\label{fig:antitree}
\end{figure}
\end{example}

\newpage
\begin{example}
A \textit{dumbbell} graph is a graph obtained by connecting two complete graphs $K_n$ and $K_m$ with a single edge. Such edge represents a ``bottleneck'' of a manifold. In general, a bottleneck of a manifold is negatively curved (i.e. in a saddle shape), and as expected, a dumbbell graph also has negative curvature around its bottleneck, as shown in Figure \ref{fig:dumbbell}.
\begin{figure}[h] 
\centering
\includegraphics[width= 14cm]{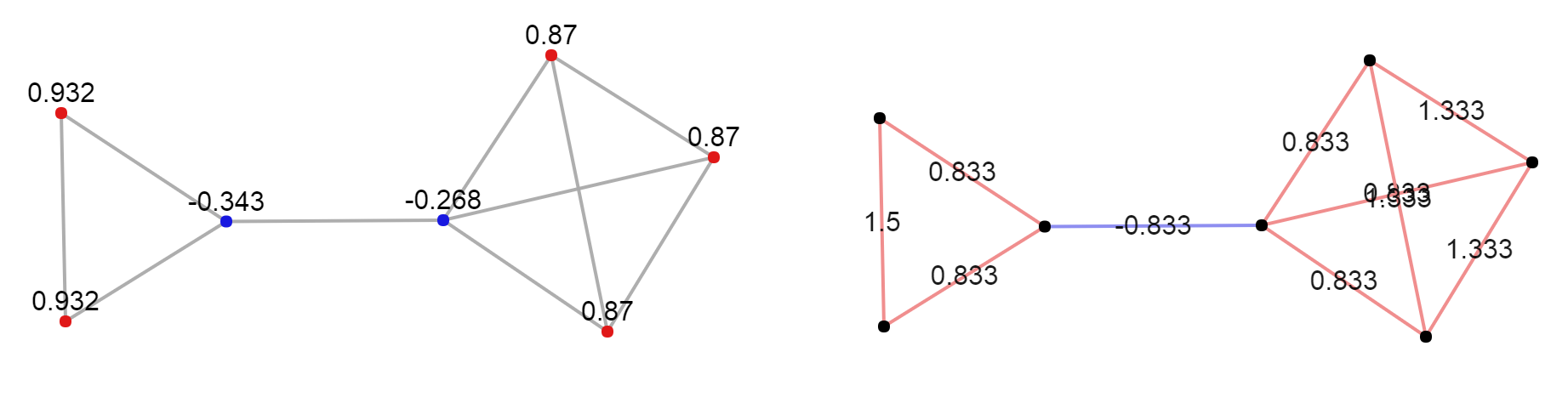}
\caption{Bakry-\'Emery and Ollivier's Ricci curvature of a dumbbell graph}
\label{fig:dumbbell}
\end{figure}
\end{example}

%******************************************************************************************

\chapter{Combinatorial Curvature} \label{CCchap}

%Gauss-Bonnet Theorem $\int_{M} K \textup{d}\lambda = 2\pi \chi(M)$ involves Euler's characteristic, which suggests the triangulation of two-dimensional surfaces into small triangles, or more generally polygons. A reverse approach is to start from \textit{tessellations} (which are, roughly speaking, tilings of polygons) and realize them as discrete surfaces. 

The idea behind Gauss-Bonnet Theorem comes from a relation between the sum of interior angles of a triangle (formed by three geodesics) on a surface and the total curvature inside that triangle (see \cite{Gauss}). 
When a surface is triangulated (i.e. partitioned into small triangles, or polygons), it resembles planar graphs. Gaussian curvature, which explains angles on surfaces, is translated into combinatorial curvature, which explains ``angles'' in planar graphs. Hence these two curvature notions describe the geometry of surfaces very similarly.

\section{Planar tessellations}

We shall start with the definition of graphs. A graph $G$, written as $G=(V,E)$ consists of a set $V$ of elements called \textit{vertices} (singular: vertex), and a set $E$ whose elements are \textit{edge}, each of which connects a pair of vertices (called endpoints of an edge). Throughout this paper, we assume graphs to be \textit{undirected}, which means edges have no direction, and to be \textit{simple}, which means they contain no loop (i.e. an edge whose endpoints are the same vertex) and no multiple edges (i.e. more than one edges sharing the same pair of endpoints). When vertices $u$ and $v$ are connected by one (and only) edge $e$ in $E$, we may say that $u$ is adjacent to $v$ (written as $u\sim v$ or $u\stackrel{e}{\sim} v$). For a vertex $v\in V$, the degree of $v$, denoted by $d_v$, is the number of vertices that are adjacent to $v$. A (finite) path is a sequence of (finite) edges which connect a sequence of all distinct vertices (except possibly the first and the last):
$$v_1\stackrel{e_1}{\sim} v_2\stackrel{e_2}{\sim} v_3\stackrel{e_3}{\sim}...\stackrel{e_{n-1}}{\sim} v_{n}.$$

The length of a path is the number of edges in its sequence. For two vertices $u$ and $v$, the combinatorial distance function $d(u,v)$ is defined to be the length of shortest path connecting $u$ and $v$. By convention, set $d(u,u)=0$ for all vertices $u$, and set $d(u,v)=\infty$ if there is no path connecting $u$ and $v$. Moreover, a graph is said to be \textit{connected} if, for every pair of vertices $u$ and $v$, there exists a path connecting $u$ and $v$.\\

In the setup of combinatorial curvature, graphs are required to be $\it{planar}$, so that the notion of \textit{faces} can be introduced. A planar graph is a graph $G=(V,E)$ that can be embedded in $\mathbb{R}^2$ without self intersecting edges. The union of edges $\bigcup_{e\in E} e$, when realized in $\mathbb{R}^2$, divides the entire space $\mathbb{R}^2$ into connected components. The closure of each component in $\mathbb{R}^2\backslash \bigcup_{e\in E} e$ is called a $\it face$. Let $F$ be the set of all faces, so we may consider it as an additional structure of a planar graph $G$: $G=(V,E,F)$. We further assume graphs to be \textit{locally finite}. A planar graph $G$ is locally finite if every point of $\mathbb{R}^2$ has an open neighborhood intersecting only finitely many faces of $G$. Local finiteness prevents graphs from clustering in arbitrarily small area.

In \cite{BP}, O. Baues and N. Peyerimhoff define conditions for locally finite planar graphs to be \textit{tessellations} as follows.

\begin{definition}[Tessellation]
A connected and locally finite planar graph $G$ is a \textit{planar tessellation}, or just \textit{tessellation}, if it satisfies the following conditions.

\begin{itemize}
\item[(i)] Every edge is contained in exactly two different faces.
\item[(ii)] Every bounded face is a \textit{polygon}: it is homeomorphic to the closed disk $\overline {\mathbb{D}}$, and its boundary is a simple cycle (i.e. a finite path in which the first and the last vertices coincide). The edges of the cycle are called \textit{sides} of the polygon.
\item[(iii)] The intersection of any two distinct faces are either an empty set, or a vertex, or an edge.   
\end{itemize}
\end{definition}

Condition (iii) suggests the convexity property for polygons. Figure \ref{fig:nontess} shows two examples where the condition (iii) breaks. 

\begin{figure}[h] 
\centering
\includegraphics[width= 10cm]{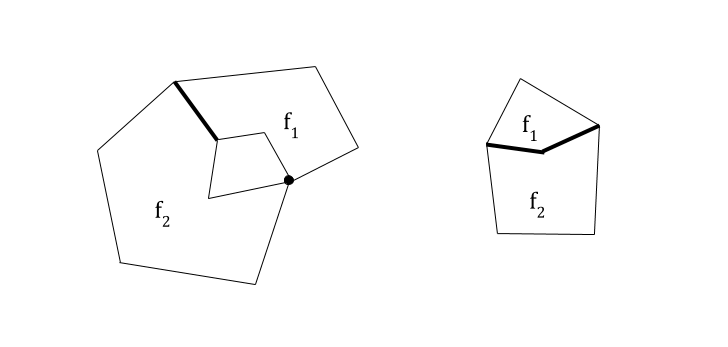}
\caption{Polygonal faces violating tessellation rule}
\label{fig:nontess}
\end{figure}

\begin{remark} \label{tess}
There are two cases of planar tessellations that we are interested in. First is an \textit{infinite tessellation}: it contains infinitely many faces, and every face is bounded. Second is a \textit{finite tessellation}: it contains exactly one unbounded face, which is homeomorphic to $\mathbb{R}^2\backslash \mathbb{D}$ and its boundary is a simple cycle. In case of a finite tessellation, by realizing $\mathbb{R}^2\cup \{\infty\}$ as a 2-dimensional sphere $S^2$, the unbounded face can be viewed as a bounded polygon (containing the point at infinity). Thus, in fact, a finite planar can be realized as a finite tessellation of $S^2$.
\end{remark}

For a face $f\in F$ of a tessellation, let $d_f$ denote its degree: the number of vertices (or equivalently, the number of sides) of the boundary of polygon $f$. The conditions on tessellations imply that $3\le d_v<\infty$, and $3\le d_f<\infty$ for all $v\in V$, $f\in F$.

Two combinatorial curvature notations are defined as follows.
\begin{definition}[Combinatorial curvature]
For each vertex $v$ and a face $f$ having $v$ as its vertex (written as $f\sim v$), the corner curvature is defined as
$$\kappa (v,f)= \frac{1}{d_v}+\frac{1}{d_f}-\frac{1}{2}.$$
For a vertex $v$, the (vertex) curvature is
$$\kappa (v)=\sum\limits_{f:f\sim v}\kappa (v,f).$$\\
summed over all faces $f$ having $v$ as their vertex. 
\end{definition}
In fact, for a fixed vertex $v$, the number of faces $f$ incident to $v$ is equal to  its degree $|v|$. Hence, the (vertex) curvature can be defined in another way as
$$\kappa (v)
=\sum\limits_{f:f\sim v} \left(\frac{1}{d_v}+\frac{1}{d_f}-\frac{1}{2}\right)= 1-\frac{d_v}{2}+\sum\limits_{f:f\sim v}\frac{1}{d_f}$$

The motivation behind this definition of curvature is ``angular defect'', which can be explained as follows. If each face $f$ were to be realized as a regular polygon of equal side length, the inner angle of polygon $f$ would be $(1-\frac{2}{d_f})\pi$ and the sum of angles of all faces $f$'s at the vertex $v$ would then be 
$$\sum\limits_{f:f\sim v} \bigg(1-\frac{2}{d_f}\bigg)\pi = 2\pi \bigg(\frac{d_v}{2}-\sum\limits_{f:f\sim v}\frac{1}{d_f}\bigg)= 2\pi (1-\kappa(v)).$$

If $\kappa(v)<0$, then the sum of angles at $v$ is more than $2\pi$, which means that these polygonal faces around $v$ form a saddle-shape surface around $v$. On the other hand, when $\kappa(v)>0$, the sum of angles at $v$ is less than $2\pi$, and therefore the point $v$ behaves like an elliptical point. In other words, the sign of $\kappa(v)$ (negative/zero/positive) corresponds to the geometry of the surface at point $v$ (hyperbolic/euclidean/spherical).

\section{Combinatorial Gauss-Bonnet and Cartan-Hadamard}

In Riemannian geometry, Gauss-Bonnet theorem states that for a closed surface $M$, the total curvature of $S$ can be related to its Euler's characteristic by the formula (see Section \ref{GBCC}): $$\int_{M} K\ \textup{d}A =2\pi \chi(M).$$  
In a case when $G$ is a finite planar tessellation, or equivalently a finite tessellation of $S^2$ (see Remark \ref{tess}), the Euler's characteristic of $G$ is given by $\chi(G)=\chi(S^2)=2$. Gauss-Bonnet theorem has the following discrete analogue for a finite planar tessellation.

\begin{theorem}[Combinatorial Gauss-Bonnet]
Let $G=(V,E,F)$ be a finite planar tessellation. Then $\sum\limits_{v\in V}\kappa(v)=2$.
\end{theorem}
\begin{proof}

\begin{align*}
\sum\limits_{v\in V}\kappa(v) 
&= \sum\limits_{v\in V} \bigg(1-\frac{d_v}{2}+\sum\limits_{f:f\sim v}\frac{1}{d_f} \bigg)\\
&=|V|-\frac{1}{2}\sum\limits_{v\in V}d_v+\sum\limits_{v\in V}\sum\limits_{f:f\sim v}\frac{1}{d_f}\\
&=|V|-|E|+\sum\limits_{f\in F}\sum\limits_{v:f\sim v}\frac{1}{d_f}\\
&=|V|-|E|+\sum\limits_{f\in F}1\\
&=|V|-|E|+|F| \\&= 2
\end{align*}
We use the fact that $\sum_{v\in V}d_v=2|E|$, since each edge is counted twice in the sum. Moreover, the order of double summations is interchangeable since the sets $V$ and $F$ are finite. Lastly, $|V|-|E|+|F|=2$ is the Euler's characteristic formula applied for a finite connected planar graph.
\end{proof}

\noindent Next, we investigate graphs that have the same sign of curvatures everywhere. Let us start with non-positively curved graphs.

\begin{corollary} \label{infcor}
A tessellation that has non-positive curvature at every vertex must be infinite. 
\end{corollary}
\begin{proof}
Follows immediately from Gauss-Bonnet formula.
\end{proof}

Next theorem is a main result from Baues and Peyerimhoff's paper \cite[Theorem 1]{BP}, which is considered as a discrete analogue of Cartan-Hadamard theorem in Riemannian geometry. We omit the proof of this theorem.

\begin{theorem}[Combinatorial Cartan-Hadamard]
Let $G=(V,E,F)$ be a tessellation. For a fixed vertex $v_0\in V$, define the \textit{cut locus} of $f_0$ to be 
$$C(v_0) := \Big\{ v'\in F: d(v_0,v') \le d(v_0,v) \textup{ for all neighbors }v \textup{ of } v' \Big\}$$

\noindent If $\kappa(v,f)\le 0$ for every corner $(v,f)$, then $G$ has no cut locus, i.e. $C(v_0)=\emptyset$ for all $v_0\in V$.
\end{theorem}
In words, the theorem asserts that, when using any vertex $v_0$ as a base point, there exists no vertex $x$ where the distance function $d_{v_0}(x):=d(v_0,x)$ attains the local maxima. Equivalently, it means that every geodesic (starting at any $v_0$) can be extended infinitely, as similarly stated in the theorem of Cartan-Hadamard (see Theorem \ref{CH}).

\section{Cheeger constant and isoperimetric inequality on graphs}

In Section \ref{GBCC}, we learn that a simply connected and complete surface $M$ with negative (sectional) curvatures uniformly bounded above by $-K_0<0$ (hence $M$ is non-compact by Cartan-Hadarmard theorem) satisfies the isoperimetric inequality:
$$\textup{area}(H) \le \frac{1}{\sqrt{K_0}} \cdot \textup{length}(\partial H).$$
\noindent for all compact surfaces $H\subset M$ with boundary $\partial H$ (see Theorem \ref{cheegernegthm}).

In graphs, Cheeger constant can be defined and the isoperimetric inequality can be read analogously as in the following definition and theorem.

\begin{definition} [Combinatorial Cheeger constant]
Let $G=(V,E)$ be a graph. For a finite subset $W\subset V$, let $\partial_E W$ be the set containing all edges which connect a vertex in $W$ to a vertex in $V\backslash W$, and define the volume of $W$ as $\mbox{vol}(W):=\sum\limits_{v\in W} d_v$. \\
\noindent The Cheeger constant is then defined to be
$$\alpha(G):= \inf_{W} \frac{|\partial_E W|}{\textup{vol}(W)}$$
where the infimum is taken over all finite subset $W\subset V$ such that $|W|\le \frac{1}{2}|V|$. Moreover,
if $G$ is infinite ($|V|=\infty$), then the constraint $|W|<\frac{1}{2}|V|$ may be removed.
\end{definition}

\begin{theorem} \label{comcheegernegthm}
Let $G=(V,E,F)$ be a planar tessellation. Suppose there exists a constant $K_0>0$ such that the corner curvature $\kappa(v,f)\le -K_0<0$ holds for every corner $(v,f)$. Then $\alpha(G) \ge 2K_0$.
\end{theorem}

\noindent A proof of this theorem with a more precise bound on $\alpha(G)$ can be found in Keller and Peyerimhoff's paper \cite[Theorem 1]{KP}.
\begin{proof}

First of all, $G$ is infinite, by Corollary \ref{infcor}. For any finite subset $W\subseteq V$, let $G_W=(W,E_W)$ denote the finite subgraph of $G$ induced by $W$, such that $E_W\subseteq E$ is the set of all edges with both endpoints in $W$. As a subgraph of a planar graph, $G_W$ is also planar, and hence inducing the set of faces, namely $F_W$. It is not always true that $F_W\subseteq F$, in particular, if the tessellation is infinite. 

This proof involves two steps. Firstly, for given any finite $W\subseteq V$, we choose a particular $W'\subseteq V$ with $\frac{|\partial_E W'|}{\mbox{vol}(W')} \le \frac{|\partial_E W|}{\mbox{vol}(W)}$. The second step is to show that our choice of $W'$ satisfies $\frac{|\partial_E W'|}{\mbox{vol}(W')} \ge 2K_0$.

\noindent\textbf{Part 1} Suppose $G_W$ has $n$ connected components, namely $G_{W_i}=(W_i,E_{W_i})$ for $i=1,\cdots,n$. Observe that $$|\partial_E W|=\sum\limits_{i=1}^n  |\partial_E W_i| \mbox{\ \ \ \ \ \ \ and\ \ \ \ \ \ \ } \mbox{vol}(W)=\sum\limits_{i=1}^n \mbox{vol}(W_i).$$

\noindent Without loss of generality, assume that $W_1$ has the minimum isoperimetric ratio:
$$c_1:=\frac{|\partial_E W_1|}{\mbox{vol}(W_1)}=\min\limits_{1\le i\le n} \frac{|\partial_E W_i|}{\mbox{vol}(W_i)}.$$ 

\noindent It follows that $$\frac{|\partial_E W|}{\mbox{vol}(W)}=\dfrac{\sum\limits_{i=1}^n  |\partial_E W_i|}{\sum\limits_{i=1}^n \mbox{vol}(W_i)} \ge c_1=\frac{|\partial_E W_1|}{\mbox{vol}(W_1)}.$$

Next, construct a set $W'\subseteq V$ by adding into the set $W_1$ all vertices $v\in V$ (if they exist) such that $v$ lies in $U$, the union of all bounded faces of $G_{W_1}$. Now consider the induced subgraph $G_{W'}=(W', E_{W'})$ with the set of faces $F_{W'}$. Geometrically, the difference between the graph $G_{W_1}$ and $G_{W'}$ is that $G_{W_1}$ was connected but may not have been a tessellation, whereas $G_{W'}$ is "simply connected" and it is a tessellation. Figure \ref{fig:comparetess} shows an example when $G_{W_1}$ has a non-polygonal face, but $G_{W'}$ has nicely tessellating faces. Note that $\partial_E {W'} \subseteq\partial_E W_1$. Hence, $$\frac{|\partial_E W_1|}{\mbox{vol}(W_1)} \ge \frac{|\partial_E W'|}{\mbox{vol}(W')}.$$ 

\begin{figure}[h] 
\centering
\includegraphics[width= 10cm]{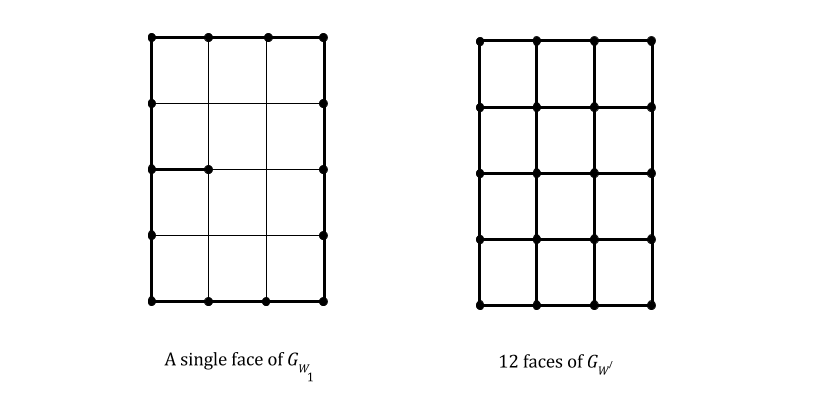}
\caption{A non-polygonal face in $G_{W_1}$ and 12 faces in $G_{W'}$}
\label{fig:comparetess}
\end{figure}

Each bounded face of $G_{W'}$ has no vertex $v\in V$ in its interior, because otherwise $v$ would be included in $W'$ in the construction step. In other words, each bounded face of $G_{W'}$ also belongs to $G$. Moreover, $G_{W'}$ has only one unbounded face. This unbounded face must not be a face of $G$; otherwise, all faces of $G_{W'}$ and all faces of $G$ coincide, which then implies that $G_{W'}=G$. This is impossible, since $G_{W'}$ is finite but $G$ is infinite. Therefore, 
$$|F_{W'}|-|F_{W'}\cap F|=1.$$

\noindent\textbf{Part 2} The assumption on $\kappa(v,f)$ implies that for any finite $W\subseteq V$,
\begin{align} \label{cc0}
\sum\limits_{v\in W} \kappa(v) \le -K_0 \mbox{vol}(W)
\end{align}
By definition of curvature,
\begin{align} \label{cc1}
\sum\limits_{v\in W} \kappa(v)&= \sum\limits_{v\in W}\bigg(1-\frac{d_v}{2}+\sum\limits_{\substack{f\in F\\f\sim v}}\frac{1}{d_f} \bigg) \nonumber\\
&=|W|-\frac{1}{2}\sum\limits_{v\in W}d_v+\sum\limits_{v\in W}\sum\limits_{\substack{f\in F\\f\sim v}}\frac{1}{d_f}
\end{align}
whereas 
\begin{equation}\label{cc2}
\sum\limits_{v\in W}d_v=2|E_W|+|\partial_E W|
\end{equation}

\noindent because each edge in $E_W$ has both endpoints in $W$ and each edge in $\partial_E W$ has exactly one endpoint in $W$.
Moreover,
\begin{equation}\label{cc3}
\sum\limits_{v\in W}\sum\limits_{\substack{f\in F\\f\sim v}}\frac{1}{d_f} \ge \sum\limits_{\substack{f\in F_W\cap F}}\sum\limits_{\substack{v\in W\\f\sim v}}\frac{1}{d_f}=\sum\limits_{\substack{f\in F_W\cap F}} 1= |F_W\cap F|
\end{equation}
\noindent since we restrict the sum to be summed only over the faces $f\in F_W\cap F$, each of which is a polygon whose vertices are in $W$.
\\

\noindent Combining \eqref{cc1}, \eqref{cc2}, and \eqref{cc3}, we obtain
\begin{align*}
\sum\limits_{v\in W} \kappa(v)\ge |W|-|E_W|-\frac{|\partial_E W|}{2}+|F_W\cup F|.
\end{align*}

\noindent In particular, for our choice of $W'\subset V$ from \textbf{Part 1} we have 
\begin{align*}
-K_0 \mbox{vol}(W') 
&\ge |W'|-|E_{W'}|-\frac{|\partial_E W'|}{2}+|F_{W'}|-1\\
&=1-\frac{|\partial_E W'|}{2}
\end{align*}
where the last equality applies Euler's formula $|W'|-|E_{W'}|+|F_{W'}|=2$ for a finite connected planar graph $W'$.
\\

\noindent We can now conclude $\frac{|\partial_E W'|}{\mbox{vol}(W')} \ge 2K_0$ as desired.
\end{proof}

\section{Combinatorial Bonnet-Myers}
At the end of \cite{Higuchi}, Higuchi conjectures that everywhere positive combinatorial curvature implies the finiteness of graphs. This conjecture can be regarded as a discrete analogue to a weak version of the Bonnet-Myers theorem. 

\begin{conjecture}[Higuchi]
A tessellation that has positive curvature at every vertex must be a finite graph. 
\end{conjecture}

Let us investigate two examples of tessellations, namely $\it prism$ and $\it antiprism$.\\
A prism is a graph with $2n$ vertices $u_1,v_1,u_2,v_2,...,u_n,v_n$ with edges joining $$u_1\sim u_2\sim ...\sim u_n \sim u_1 \mbox{ and } v_1\sim v_2\sim ...\sim v_n \sim v_1 \mbox{ and } u_i\sim v_i$$ for all $1\le i\le n$. Its faces consist of two $n$-gons, and $n$ quadrilaterals. See Figure \ref{fig:prism}.

\begin{figure}[h] 
\centering
\includegraphics[width= 10cm]{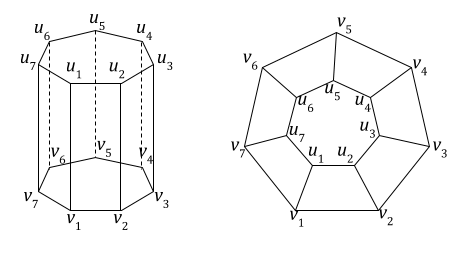}
\caption{Prism and its embedding in $\mathbb{R}^2$}
\label{fig:prism}
\end{figure}

In its embedding in $\mathbb{R}^2$, the unbounded component represents one of the two $n$-gonal faces. For every vertex $v$ of the prism, there are one $n$-gon and two quadrilaterals incident to it. Hence the combinatorial curvature can be calculated by $$\kappa(v)=1-\frac{3}{2}+\bigg(\frac{1}{4}+\frac{1}{4}+\frac{1}{n}\bigg)= \frac{1}{n} > 0.$$

An antiprism can be constructed from a prism with additional edges $u_i\sim v_{i+1}$ for $1\le i \le n-1$ and $u_n\sim v_1$. It has two $n$-gonal faces, and  $2n$ triangular faces. See Figure \ref{fig:antiprism}.

\begin{figure}[h] 
\centering
\includegraphics[width= 10cm]{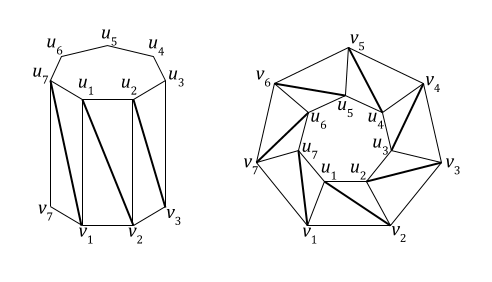}
\caption{Antiprism and its embedding in $\mathbb{R}^2$}
\label{fig:antiprism}
\end{figure}

Each vertex $v$ of an antiprism has one $n$-gon and three quadrilaterals incident to it, so the curvature is $$\kappa(v)=1-\frac{4}{2}+\bigg(\frac{1}{3}+\frac{1}{3}+\frac{1}{3}+\frac{1}{n}\bigg)= \frac{1}{n} > 0.$$

As shown above, prism and antiprism demonstrate two classes of tessellations that have positive curvature everywhere. Although both prism and antiprism are finite graphs, their numbers of vertices can be arbitrary large.
In \cite{Mohar}, DeVos and Mohar proved Higuchi's conjecture and provided a further insight about the finiteness that: all everywhere-positively-curved tessellation (except prisms and antiprisms) have a uniform upper bound on the number of their vertices, and they asked for a sharp bound. In \cite{208}, the authors gives an example of one such graph with $208$ vertices. On the other hand, it was recently proved in \cite{Ghidelli} that all such graphs have at most 208 vertices, hence 208 is the optimal number.

\begin{theorem}[Combinatorial Bonnet-Myers]
If $\kappa(v)>0$ for every vertex $v$ in a planar graph $G=(V,E)$, then $G$ is finite. Moreover, $G$ is either a prism, antiprism, or $|V|\le 208$. 
\end{theorem}

\chapter{Bakry-\'{E}mery Curvature} \label{BECchap}

While the previous curvature notion was based on a discrete version of the Gauss-Bonnet Theorem in two-dimension, the curvature notion in this chapter, introduced by D. Bakry and M. \'{E}mery \cite{BE}, is based on Bochner's formula from Riemannian geometry. Graphs are no longer assumed to be planar, and their dimensions are not restricted to two. Instead, the dimension can be chosen to be an arbitrary positive real number, including $\infty$.

\section{CD inequality and $\Gamma$-calculus}
Bochner's formula states that for every smooth real function $f\in C^\infty(M)$ and at every point $x\in M$,
\begin{equation} \label{bochner}
\frac{1}{2}\Delta |\grad f|^2 (x)= \|\textup{Hess }f\|^2(x) + \langle \grad\Delta f(x), \grad f(x)\rangle +\mbox{Ric}(\grad f(x)),
\end{equation}

\noindent Further, defined at each point $x$ the curvature term $K_x:= \inf\limits_{v\in T_xM} \frac{\textup{Ric}(v)}{|v|^2}$ which gives a lower bound for Ricci curvature term: $$\textup{Ric}(\grad f(x))\ge K_x|\grad f(x)|^2.$$ Recall also Proposition \ref{HL}: $\displaystyle \|\textup{Hess }f\|^2(x)\ge \frac{1}{n}(\Delta f(x))^2$. Combining these two inequalities into the equation \eqref{bochner}, we have the so called "curvature-dimension" inequality,

\begin{equation} \label{CD}
\frac{1}{2}\Delta |\grad f|^2 (x) -  \langle \grad\Delta f(x), \grad f(x)\rangle \ge \frac{1}{n}(\Delta f(x))^2 +K_x|\grad f(x)|^2
\end{equation}

According to \cite{BE}, define bilinear operators $\Gamma$ and $\Gamma_2$ as follows.
\begin{definition}[$\Gamma$ and $\Gamma_2$ notions]
For $f,g\in C^\infty(M)$, define
\begin{equation} \label{Gamma}
2\Gamma(f,g)(x) := \Delta(fg)(x) - f(x)\cdot\Delta g(x)-\Delta f(x) \cdot g(x)
\end{equation}

\begin{equation} \label{Gamma2}
2\Gamma_2(f,g)(x) := \Delta\Gamma(f,g)(x) - \Gamma(f,\Delta g)(x)-\Gamma(\Delta f, g)(x)
\end{equation}
\end{definition}
%\begin{proposition}
%In fact, $\Gamma(f,g)(x)=2\langle \grad f(x), \grad g(x) \rangle$
%\end{proposition}
%\begin{proof}
%Apply the product rules for operators $\dive$and $\grad$.
%\begin{align*}
%\Delta(fg)&=\dive(\grad(fg))\\
%&=\dive(f\cdot\grad g)+\dive(g\cdot\grad f)\\
%&=\langle \grad f, \grad g\rangle +f\cdot\dive(\grad g) + \langle \grad g, \grad f\rangle +g\cdot\dive(\grad f)\\
%&=2\langle \grad f, \grad g\rangle+f\cdot\Delta g+g\cdot\Delta f
%\end{align*}
%\noindent yielding the desired result.
%
%\end{proof}
\noindent In fact, $\Gamma(f,g)(x)=2\langle \grad f(x), \grad g(x) \rangle$ by the product rule (see Proposition \ref{PR}).
\noindent We further denote $\Gamma(f) := \Gamma(f,f)$, and $\Gamma_2(f) := \Gamma_2(f,f)$. The curvature-dimension inequality \eqref{CD} can then be rewritten as 
\begin{equation}
\Gamma_2(f)(x) \ge \frac{1}{n}(\Delta f(x))^2 +K_x\Gamma(f)(x)
\end{equation}
which holds for all $f\in C^\infty(M)$ and $x\in M$.\\

Observe that the above curvature-dimension inequality involves the $\Gamma$ and $\Gamma_2$ terms, which were defined merely via the Laplacian. This allows us to consider curvature-dimension property on any space, once the Laplacian is specified on such space.
\begin{definition}
Let $X$ be a space and $C(X)$ be the function space of $X$, that is the set of all functions $f:X\rightarrow \mathbb{R}$, equipped with the addition and scalar multiplication rules: $(f+f')(x)=f(x)+f'(x)$ and $(cf)(x)=c\cdot f(x)$. Assume that the $X$ has Laplacian operator $\Delta$ defined on it. Fix a number $n\in \mathbb{R}^+\cup \{\infty\}$ to be the \textit{dimension} of $X$. The curvature at each point $x\in X$ is defined to be the maximal number $K_x$ such that the inequality

\begin{equation}
\Gamma_2(f)(x) \ge \frac{1}{n}(\Delta f(x))^2 +K_x\Gamma(f)(x)
\end{equation}
\noindent holds true for all functions $f\in C(X)$\\\\
Moreover, for a fixed real number $K$, we say that $X$ satisfies $CD(K,n)$ if $K_x\ge K$ for all $x\in X$; in other words,
\begin{equation}
\Gamma_2(f)(x) \ge \frac{1}{n}(\Delta f(x))^2 +K\Gamma(f)(x)
\end{equation}
holds for all $x\in X$ and for all $f\in C(X)$.\\
Here the operators $\Gamma$ and $\Gamma_2$ on $X$ are also defined as in the equation \eqref{Gamma} and \eqref{Gamma2}.
\end{definition}

\noindent In particular, Laplacian on graphs can be specified as follows.
\begin{definition}[Discrete Laplacian]
Let $G=(V,E)$ be a graph with finite degree on each vertex. Laplacian $\Delta:C(V)\rightarrow C(V)$ is a linear operator, defined on any function $f: V\rightarrow \mathbb{R}$ as
$$\Delta f(x) := \frac{1}{d_x}\sum\limits_{z\sim x} \bigg(f(z)-f(x)\bigg)$$
for all vertices $x\in V$.
\noindent In terms of matrix representation, we can write Laplacian as $$\Delta=D^{-1}(A-\textup{Id})$$
\noindent where $D$ is the diagonal matrix whose entries are the vertex degrees: $D_{xx}=d_x$, and $A$ is the agjacency matrix: $A_{xy}=1$ if $x\sim y$ and 0 otherwise.
\end{definition}

\noindent This notion is sometimes called the normalized Laplacian (in contrast to the non-normalized one, where the factor $\frac{1}{d_x}$ is dropped). Here are some properties of operators $\Delta$ and $\Gamma$ defined on graphs.
\begin{proposition} \label{proplap}
Let $G=(V,E)$ be a graph. For $f,g\in C(V)$, we have
\begin{itemize}
\item[(a)]
$$2\Gamma(f,g)(x)=\frac{1}{d_x}\sum\limits_{z\sim x} \bigg(f(z)-f(x)\bigg)\bigg(g(z)-g(x)\bigg).$$
In particular,
$$2\Gamma(f)(x)=\frac{1}{d_x}\sum\limits_{z\sim x} \bigg(f(z)-f(x)\bigg)^2.$$
\item[(b)]
$(\Delta f(x))^2 \le 2\Gamma f(x)$ for all $f\in C(V)$ and $x\in V$.
\item[(c)] If $G$ is finite, then $$\sum\limits_{x\in V} d_x \Delta f(x)=0$$ \noindent for all $f\in C(V)$.
\end{itemize}
\end{proposition}
\begin{proof} \mbox{  }
\begin{itemize}
\item[(a)] Straightforward calculation from the definition gives
\begin{align*}
2\Gamma(f,g)(x)&= \Delta(fg)(x) - f(x)\cdot\Delta g(x)-\Delta f(x) \cdot g(x)\\
&=\frac{1}{d_x}\sum\limits_{z\sim x} \bigg[\bigg(f(z)g(z)-f(x)g(x)\bigg) \\
&\ \ \ \ \ \ \ \ \ \ \ \ \ \ \ - f(x)\bigg(g(z)-g(x)\bigg) -g(x)\bigg(f(z)-f(x)\bigg)\bigg]\\
&=\frac{1}{d_x}\sum\limits_{z\sim x} \bigg(f(z)-f(x)\bigg)\bigg(g(z)-g(x)\bigg).
\end{align*}
\noindent The second identity in part (a) follows immediately.\\
\item[(b)]
From Arithmetic-Quadratic mean (AM-QM) inequality,
$$|\Delta f(x)| \le \frac{1}{d_x}\sum\limits_{z\sim x} \bigg|f(z)-f(x)\bigg| \le \sqrt {\frac{1}{d_x}\sum\limits_{z\sim x} \bigg(f(z)-f(x)\bigg)^2} \stackrel{(a)}{=} \sqrt{2\Gamma f(x)}.$$

\item[(c)]
Suppose $G$ is finite.
\begin{align*}
\sum\limits_{x\in V} d_x \Delta f(x) &= \sum\limits_{x\in V} \sum\limits_{z: z\sim x}\bigg(f(z)-f(x)\bigg) \\ 
&=\sum\limits_{x\in V} \sum\limits_{z: z\sim x} f(z) - \sum\limits_{x\in V} d_xf(x)\\
&=\sum\limits_{z\in V} \sum\limits_{x: z\sim x} f(z) - \sum\limits_{x\in V} d_xf(x)\\
&=\sum\limits_{z\in V} d_z f(z) - \sum\limits_{x\in V} d_xf(x) = 0
\end{align*}
\end{itemize}
\end{proof}

In Section \ref{eigenlap}, we obtain Lichnerowicz's bound on first nonzero eigenvalue by taking integral on the Bochner's formula and applying the Divergence theorem. Here, we imitate a similar result in a discrete analogue. 

\begin{theorem}[B-E Lichnerowicz]
Let $G=(V,E)$ be a finite connected graph satisfying $CD(K,n)$ condition for some $K>0$. Then the first nonzero eigenvalue with respect to the Laplacian operator $\Delta$ satisfies $\lambda_1 \ge \frac{n}{n-1}K$.
\end{theorem}
In fact, the condition that $G$ is finite can be removed, since the condition $CD(K,n)$ when $K>0$ already implies the finiteness of $G$ by Bonnet-Myer's theorem, which we will discuss later on in this chapter.
\begin{proof}
Suppose $f$ is an eigenfunction satisfying $\Delta f+\lambda f=0$. Due to a scalar multiplication to $f$, we may assume $\sum\limits_{x\in V} d_x  f^2(x)=1$.
Now we aim to compute the total sum of all terms in the $CD(K,n)$ condition: $$\Gamma_2(f)(x) \ge \frac{1}{n}(\Delta f(x))^2 +K\Gamma(f)(x).$$ 

\noindent From the definition of $\Gamma(f)$, we have
\begin{align*}
2\cdot\sum\limits_{x\in V} d_x \Gamma(f)(x)
&=\sum\limits_{x\in V} d_x \Delta (f^2)(x)-2\sum\limits_{x\in V} d_x f(x)\Delta f(x)\\
&=0+2\lambda\sum\limits_{x\in V} d_x f^2(x)\\
&=2\lambda
\end{align*}
Here we used the fact that $\sum\limits_{x\in V} d_x \Delta (f^2)(x)=0$ due to the discrete Divergence theorem (Proposition \ref{proplap}(c)) applied to the function $f^2$, and the fact that $\Delta f=-\lambda f$. Therefore, the total sum of $\Gamma(f)$ is
\begin{equation} \label{totalsumgamma}
\sum\limits_{x\in V} d_x \Gamma(f)(x)=\lambda.
\end{equation}

\noindent Similarly, the total sum of $\Gamma_2(f)$ can be calculated as
\begin{align*}
2\cdot\sum\limits_{x\in V} d_x \Gamma_2(f)(x)
&=\sum\limits_{x\in V} d_x \Delta \Gamma(f)(x)-2\sum\limits_{x\in V} d_x \Gamma(f,\Delta f)(x)\\
&=0-2\sum\limits_{x\in V} d_x \Gamma(f,-\lambda f)(x)\\
&=2\lambda \sum\limits_{x\in V} d_x \Gamma(f)(x)\\
&\stackrel{\eqref{totalsumgamma}}{=}2\lambda^2.
\end{align*}
Therefore,
\begin{equation} \label{totalsumgamma2}
\sum\limits_{x\in V} d_x \Gamma_2(f)(x)=\lambda^2.
\end{equation}

\noindent Moreover, the total sum of $(\Delta f)^2$ is simply
\begin{equation} \label{total3}
\sum\limits_{x\in V} d_x (\Delta f)^2(x)=\sum\limits_{x\in V} d_x (-\lambda f)^2(x)=\lambda^2.
\end{equation}

Combining equations \eqref{totalsumgamma}, \eqref{totalsumgamma2}, and \eqref{total3} into $CD(K,n)$, we obtain $$\lambda^2\ge \frac{1}{n}\lambda^2+K\lambda.$$
\noindent Therefore, $\lambda_1\ge \frac{n}{n-1}K$ as desired.
\end{proof}

\section{Motivation of the defined Laplacian in graphs}

We have seen Laplacian in $\mathbb{R}^n$ with Euclidean metric. In particular when $n=2$, 
$$\Delta f(x) = \frac{\partial^2 f}{\partial x_1^2}(x)+\frac{\partial^2 f}{\partial x_2^2}(x).$$
Express the derivatives in terms of finite differences, 
\begin{align*}
\Delta f(x) =\lim\limits_{h\rightarrow 0} \dfrac{f(x+he_1)+f(x-he_1)+f(x+he_2)+f(x-he_2)-4f(x)}{h^2}
\end{align*}
By discretizing $\mathbb{R}^2$ as $\mathbb{Z}^2$ and set $h=1$, the discrete Laplacian then becomes
\begin{align*}
\Delta f(x) 
&= f(x+e_1)+f(x-e_1)+f(x+e_2)+f(x-e_2)-4f(x)\\
&= \sum\limits_{z\sim x}\Big( f(z)-f(x)\Big)
\end{align*}
\noindent as we treat $x\pm e_i$'s to be the neighbors of $x$ (see Figure \ref{fig:z2laplace}).
\begin{figure}[h] 
\centering
\includegraphics[width= 5cm]{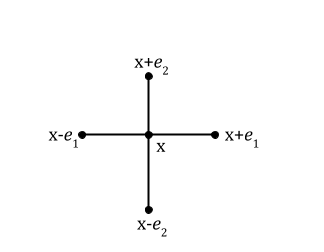}
\caption{Neighbors of $x$ in $\mathbb{Z}^2$}
\label{fig:z2laplace}
\end{figure}

\section{Heat semigroup operator}
In this section, we introduce another operator, namely \textit{heat semigroup} operator, which will be a useful tool in the proof of Bonnet-Myers later on in this chapter.

\begin{definition}
Let $X$ be a space with Laplacian operator $\Delta$.
For $t\in [0,\infty)$, a \textit{heat semigroup} operator $P_t: C(X)\rightarrow C(X)$ is defined by 
$$P_t:=e^{t\Delta}$$ %= \sum\limits_{n=0}^{\infty} \frac{t^n\Delta^n}{n!}$$ 
for all $f\in C(X)$.
\end{definition}

The operator $P_t$ is differentiable in $t$, and its derivative satisfies 
$$\frac{\partial}{\partial t} P_t=\Delta P_t.$$

\noindent Basic properties of $P_t$ are listed in the following proposition

\begin{proposition} \label{heatprop}
Let $P_t$ be the heat semigroup operator defined as in above. Then
\begin{itemize}
\item[(a)] $\Delta P_t = P_t\Delta$
\item[(b)] $f\ge 0$ implies $P_tf\ge 0$
\item[(c)] $||P_t(f)||_\infty \le ||f||_\infty$
\end{itemize}
\end{proposition}

Although this proposition holds in great generality, we will prove it here in the case of normalized Laplacian $\Delta$ on finite graphs.
\begin{proof}
Recall from the definition that $\Delta=D^{-1}(A-\textup{Id})$, which is a bounded operator, so we may write $e^{t\Delta}=\sum\limits_{n=0}^{\infty} \frac{t^n\Delta^n}{n!}$, and thus $$\Delta P_t=\sum\limits_{n=0}^{\infty} \frac{t^n\Delta^{n+1}}{n!}=P_t\Delta$$ \noindent which gives (a). 

Note that $B:=\textup{Id}+D^{-1}(A-\textup{Id})$ is also a bounded operator, and it has all entries nonnegative. Therefore, $$e^{t\Delta}=e^{-t}e^{tB}=e^{-t}\sum\limits_{n=0}^{\infty} \frac{t^nB^n}{n!}$$ also have all of its entries nonnegative. For a function $f\in C(V)$ such that $f\ge 0$, it is represented by a column vector $f$ whose entries are nonnegative. Thus $P_tf= e^{t\Delta}f$ has all entries nonnegative, meaning $P_t(f)\ge 0$.

Lastly, let $F=\max\limits_{x\in V} |f(x)|=\|f\|_\infty$, thus $F\mathbbm{1}-f \ge 0$. Part (b) then implies $$0\le P_t(F\mathbbm{1}-f)=FP_t\mathbbm{1}-P_tf=F\mathbbm{1}-P_tf,$$
\noindent that is $P_tf\le F$ as desired.
\end{proof}

In \cite{GL}, Gong and Lin prove that the condition $CD(K,\infty)$ can be characterized in term of $P_t$ as in the following theorem. This theorem serves as a part in the proof of Bonnet-Myer's theorem. 
%**Give a brief introduction about Ledoux-Bakry theorem. [details in Gong, Lin, 2015]

\begin{theorem}\label{LB}
If $G=(V,E)$ satisfies $CD(K,\infty)$ condition, then
\begin{align} \label{LBeq}
\Gamma(P_tf)(x) \le e^{-2Kt}P_t(\Gamma f)(x)
\end{align}
\noindent for all $x\in V$ and all bounded $f:V\rightarrow \mathbb{R}$.
\end{theorem}

\begin{proof}
Fix $x\in V$ and $t\in [0,\infty)$. Define a real function $F(s)$ on $s\in [0,t]$ as 
\begin{equation*}
F(s):=e^{-2Ks}P_s(\Gamma P_{t-s} f) (x)
\end{equation*}

Note that $F(0)=\Gamma (P_t f) (x)$ and $F(t)=e^{-2Kt}P_t\Gamma f(x)$ are the terms on left-hand side and right-hand side of the inequality \eqref{LBeq}. We need $F(0)\le F(t)$, so it suffices to prove that $F'(s)\ge 0$ for all $0<s<t$.\\

\noindent Product rule and chain rule of the differentiation give
\begin{align*}
F'(s)=e^{-2Ks}\bigg[-2K P_s(\Gamma P_{t-s}f)(x) + (\frac{\partial}{\partial s} P_s)(\Gamma P_{t-s}f)(x) + P_s(\frac{\partial}{\partial s}\Gamma(P_{t-s}f))(x)\bigg]
\end{align*}
With the relation $\frac{\partial}{\partial s} P_s = \Delta P_s = P_s\Delta$ substituted into the second term in the bracket above, we can then pull out $P_s$ and obtain
$F'(s)=e^{-2Ks} P_s h(x)$ where $h$ denotes the operator
\begin{align*}
h:= -2K \Gamma P_{t-s}f + \Delta(\Gamma P_{t-s}f) + \frac{\partial}{\partial s}\Gamma(P_{t-s}f)
\end{align*}
Moreover, in the last term, 
\begin{align*}
\frac{\partial}{\partial s}\Gamma(P_{t-s}f) 
&=\frac{\partial}{\partial s}  \Gamma(P_{t-s}f, P_{t-s}f)\\
&=2\Gamma\Big(\frac{\partial}{\partial s}(P_{t-s}f), P_{t-s}f\Big)\\
&=-2\Gamma(\Delta P_{t-s}f, P_{t-s}f)\\
&\stackrel{\makebox[0pt]{\small def}}{=} 2\Gamma_2(P_{t-s}f)-\Delta(\Gamma P_{t-s}f)
\end{align*}

\noindent Hence, 
\begin{align*}
h= -2K \Gamma P_{t-s}f + 2\Gamma_2(P_{t-s}f) \ge 0
\end{align*}
due to the condition $CD(K,\infty)$. The proposition \eqref{heatprop} implies that $P_s h \ge 0$, which gives $F'(s)\ge 0$ as desired.

\end{proof}

\section{Bakry-\'Emery Bonnet-Myers}

Bonnet-Myers in the sense of Bakry-\'Emery states that a graph with strictly positive Bakry-\'Emery curvature bounded away from zero must be a finite graph, and the bound of diameter can be estimated in term of curvature.
Here, we give a proof in case of $\infty$-dimension. The theorem also holds for any dimension $n<\infty$ but with a different bound on diameter (see \cite[Theorem 2.4]{LMP}).

\begin{theorem}[B-E Bonnet-Myers]
Let $G$ be a connected graph satisfying $CD(K,\infty)$ condition for some $K>0$. Then $G$ is finite and $$\textup{diam } G \le \frac{2}{K}$$ \end{theorem}

\begin{proof}
Consider arbitrary vertices $x_0,y_0\in V$, and let $L=d(x_0,y_0)$ be the length of shortest path(s) connecting $x_0$ and $y_0$.   Define a function $f:V\rightarrow \mathbb{R}$
\begin{align*}
f(x):= max\{ L-d(x_0,x), 0\},
\end{align*}
\noindent so that $f$ is bounded, and that $f(x_0)=L$ and $f(y_0)=0$. 

\noindent By triangle inequality, 
\begin{align*}
L&=|f(x_0)-f(y_0)| \\
&\stackrel{\triangle}{\le} |f(x_0)- P_tf(x_0)| + |P_tf(x_0)-P_tf(y_0)|+ |P_tf(y_0)-f(y_0)|
\end{align*}
holds for all $t>0$.

\noindent The next two steps are to prove that $|f(x)- P_tf(x)| \le \frac{1}{K}$ holds for any $x$, and that $|P_tf(x_0)-P_tf(y_0)| \rightarrow 0$ as $t\rightarrow \infty$. This will guarantee $L\le \frac{2}{K}$.
\begin{itemize}
\item
First, the fundamental theorem of calculus gives
\begin{align*}
|f(x)- P_tf(x)| \le \int\limits_0^t \Big|\frac{\partial}{\partial s}P_sf(x)\Big| ds 
= \int\limits_0^t \Big| \Delta P_sf(x)\Big| ds
\end{align*}
where
\begin{equation*}
\Big| \Delta P_sf(x)\Big| \le \sqrt{2\Gamma(P_s f)(x)}
\le e^{-Ks} \sqrt{P_s(2\Gamma f)(x)} \le e^{-Ks} \sqrt{2||\Gamma f||_\infty} \le e^{-Ks} 
\end{equation*} by Proposition \ref{proplap}(b), Theorem \ref{LB}, and Proposition \ref{heatprop}(c), and
because $2\Gamma f(z) =\frac{1}{d_{z}}\sum\limits_{y\sim z} (f(y)-f(z))^2\le \frac{1}{d_{z}}\sum\limits_{y\sim z} 1 = 1$ for all $z\in V$.\\\\
Therefore, 
\begin{align*}
|f(x)- P_tf(x)| \le \int\limits_0^t e^{-Ks} ds \le \frac{1}{K} 
\end{align*}
\item
Second, it suffices to prove that $|P_tf(x)-P_tf(z)| \rightarrow 0$ as $t\rightarrow \infty$ for any neighboring vertices $x\sim z$, and then using again the triangle inequality to deal with vertices at longer distance.

\noindent As in the previous part, 
\begin{align*}
e^{-Kt} \ge \sqrt{2\Gamma(P_t f)(x)}  &\stackrel{\mbox{def}}{=} \sqrt{\frac{1}{d_x}\sum\limits_{y: y\sim x}|P_tf(y)-P_tf(x)|^2} \\
&\ge \frac{1}{\sqrt{d_x}} |P_tf(x)-P_tf(z)|
\end{align*}
holds for any $z\sim x$.
Taking $t\rightarrow \infty$, we obtain $|P_tf(x)-P_tf(z)|\rightarrow 0$ as desired.
\end{itemize}
\end{proof}

\chapter{Ollivier's Ricci Curvature} \label{ORCchap}

Ollivier's Ricci curvature notion is motivated from the ``phenominon'' (which is the exact word that Ollivier chose to describe in  \cite[pp.~4]{Ollivier}) that Ricci curvature determines whether the average distance of two balls around $x$ and $y$ is larger or smaller than the distance between $x$ and $y$ (see Section \ref{avgsec}). Ollivier regards this average distance of two balls as ``transportation distance'' between two measures. We shall start with the concept of transportation distance (namely Wasserstein distance). Alternatively, see Villani's \cite{Villani} for a broader introduction to this topic.

\begin{definition}[Transport plan]
Let $G=(V,E)$ be a locally-finite and connected graph. Let $$P(V):=\left\{\mu:V\rightarrow [0,1] \ \bigg| \ \sum\limits_{x\in V}\mu(x)=1, \textup{ and } \mu(x)>0 \textup{ for finitely many } x\textup{'s} \right\}$$ be the space of all probability measures on $V$ with finite supports. 

Given any $\mu,\nu \in P(V)$, a \textit{transport plan} (or, in short, a \textit{plan}) from $\mu$ to $\nu$ is a function $\pi: V\times V \rightarrow [0,1]$ satisfying 
$$\mu(x) = \sum\limits_{y\in V} \pi(x,y) \ \ \ \ \mbox{ and }\ \ \ \  \nu(y)=\sum\limits_{x\in V} \pi(x,y).$$

\noindent Furthermore, define $\prod(\mu,\nu)$ to be the set of all transport plans from $\mu$ to $\nu$, and the (transportation) cost of a plan $\pi$ is given by $$\textup{cost}(\pi)=\sum\limits_{x,y\in V} d(x,y)\pi(x,y).$$

%define the support of a plan $\pi\in \prod(\mu,\nu)$ to be $$\textup{supp }\pi := \{(x,y)\in V\times V: \ \pi(x,y)>0\}.$$
\end{definition}

In words, $\pi(x,y)$ represents the amount of mass being transport from a vertex $x$ to a vertex $y$ according to the plan $\pi$. The cost for transporting per unit mass is the (combinatorial) distance function $d$.

\begin{definition}[Wasserstein metric/distance] 
Given any $\mu,\nu \in P(V)$, the Wasserstein metric is a function $W_1: P(V)\times P(V) \rightarrow [0,\infty)$ defined by
\begin{equation} \label{W1}
W_1(\mu,\nu) := \inf\limits_{\pi\in\prod(\mu,\nu)} \sum\limits_{x,y\in V} d(x,y)\pi(x,y).
\end{equation}
\end{definition}

A plan $\pi$ that yields the infimum in $\eqref{W1}$ is called \textit{optimal transport plan}.

\begin{proposition}
$W_1$ is indeed a distance function on $P(V)$, i.e. it satisfies the metric axioms.
\end{proposition}
\begin{proof}
	
Symmetry property $W_1(\mu,\nu)=W_1(\nu,\mu)$ is obvious: for every transport plan $\pi\in\prod(\mu,\nu)$, the ``reverse'' plan $\pi^{(-1)}\in\prod(\nu,\mu)$ defined by $$\pi^{(-1)}(x,y):=\pi(y,x)$$ for all $x,y$ costs the same as the plan $\pi$. Identity property: $W_1(\mu,\nu)=0\ \Rightarrow \mu=\nu$ can also be proved easily by contraposition. The only non-trivial property to be checked is the triangle inequality:
$$W_1(\mu,\nu)+W_1(\nu,\rho) \ge W_1(\mu,\rho)$$ for all probability measures $\mu,\nu,\rho\in P_1(V)$.

Assume $\pi_1\in\prod(\mu,\nu)$ and $\pi_2\in\prod(\nu,\rho)$. Construct $\pi_3:V\times V\rightarrow[0,1]$ to be $$\pi_3(x,z):=\sum\limits_{y}^* \frac{\pi_1(x,y)\pi_2(y,z)}{\nu(y)}$$
for all $x,z\in V$. Here $\sum\limits_y^*$ means the sum is taken over $y$ such that $\nu(y)\not=0$.\\
Then observe that 
\begin{align*}
\sum\limits_{z}\pi_3(x,z) &= \sum\limits_{z}\sum\limits_y^* \frac{\pi_1(x,y)\pi_2(y,z)}{\nu(y)}\\
&=\sum\limits_y^* \frac{\pi_1(x,y)}{\nu(y)}\bigg(\sum\limits_{z} \pi_2(y,z)\bigg)\\
&=\sum\limits_y^*  \pi_1(x,y)  && (\pi_2\in\prod(\nu,\rho))\\
&=\sum\limits_y  \pi_1(x,y)=\mu(x) && (\pi_1\in\prod(\mu,\nu))
\end{align*} because $\pi_1(x,y)=0$ for all $x$ and $y$ such that $\nu(y)=0$.

Similarly, one can check that $\sum\limits_{x}\pi_3(x,z)= \rho(z)$. Thus $\pi_3\in \prod(\mu,\rho).$ 
Moreover, the total cost of $\pi_3$ is less than or equal to the cost of $\pi_1$ and the cost of $\pi_2$ combined:
\begin{align*}
\sum\limits_{x,z} d(x,z)\pi_3(x,z)
&=\sum\limits_{x,z}\sum\limits_y^* d(x,z)\cdot\frac{\pi_1(x,y)\pi_2(y,z)}{\nu(y)}\\
&\stackrel{\triangle}{\le}\sum\limits_{x,z}\sum\limits_y^* \bigg(d(x,y)\frac{\pi_1(x,y)\pi_2(y,z)}{\nu(y)}+ d(y,z)\frac{\pi_1(x,y)\pi_2(y,z)}{\nu(y)}\bigg)\\
&=\sum\limits_{x}\sum\limits_{y}^*d(x,y)\pi_1(x,y)+ \sum\limits_{z}\sum\limits_{y}^*d(y,z)\pi_2(y,z)\\
&=\sum\limits_{x,y}d(x,y)\pi_1(x,y)+\sum\limits_{y,z}d(y,z)\pi_2(y,z)
\end{align*}

By considering over all $\pi_1\in\prod(\mu,\nu)$ and $\pi_2\in\prod(\nu,\rho)$, we can then conclude $W_1(\mu,\rho)\le W_1(\mu,\nu)+W_1(\nu,\rho)$.
\end{proof}

Wasserstein distance $W_1(\mu,\nu)$ represents a minimal total cost (when considered among all possible plans) of transporting masses which are distributed as in $\mu$ to masses which are distributed as in $\nu$. The subscript $1$ in $W_1$ indicates that the cost function is $d^1$. 

In general, calculating $W_1(\mu,\nu)$ directly by finding an optimal transport plan $\pi$ can be very difficult. An easier alternative method is via the following Kantorovich Duality Theorem (see \cite[pp.~19]{Villani}, or alternatively see further discussion in Section \ref{sect:LP}).

\begin{theorem}[Kantorovich Duality] \label{duality}
\begin{equation} \label{dual}
\inf\limits_{\pi\in\prod(\mu,\nu)} \sum\limits_{x,y\in V} d(x,y)\pi(x,y) 
= \sup\limits_{\Phi\in \textup{1-Lip}} \sum\limits_{x\in V} \Phi(x)\bigg(\mu(x)-\nu(x)\bigg)
\end{equation}
\noindent where $\textup{1-Lip}=\Big\{ f\in C(V)\ \Big| \ |f(x)-f(y)|\le d(x,y)\Big\}$ is the space of all Lipschitz continuous functions on $V$ with Lipschitz constant 1.
Such 1-Lipschitz function $\Phi$ yielding the maximum is called an optimal Kantorovich potential.
\end{theorem}

The method is to find a plan $\pi\in \prod(\mu,\nu)$ and a function $\phi\in\textup{1-Lip}$ such that 
\begin{equation} \label{dualcheck}
\sum\limits_{x,y\in V} d(x,y)\pi(x,y) 
=\sum\limits_{x\in V} \Phi(x)\bigg(\mu(x)-\nu(x)\bigg).
\end{equation}

Then Duality Theorem asserts that such $\pi$ and $\phi$ are an optimal transport plan and an optimal Kantorovich potential, respectively, and the terms in \eqref{dualcheck} must have the value of $W_1(\mu,\nu)$. An explicit calculation of $W_1(\mu,\nu)$ will be shown in Example \ref{ex: Kcalc}.

\section{Definition of Ollivier's Ricci curvature}

Let $G=(V,E)$ be a graph. Consider a transition matrix $P$ defining a lazy simple random walk on $G$ with the probability $p$ to stay unmoving at any vertex (and hence $p$ is called $\textup{idleness}$ parameter) and equal probability to move to any one of its neighbor. In other words, the probability of moving from $x$ to $y$ in one-time step is
\begin{align*}
P_{xy}:=\begin{cases} 
p  &,\  y=x\\ 
\frac{1-p}{d_x}  &,\ y\sim x\\
0&,\ \mbox{otherwise.} \end{cases}
\end{align*}

\begin{definition}
Let $G=(V,E)$ be a graph. For any vertex $x\in V$, let the measure $\delta_x\in P_1(V)$ be the Dirac measure, that is,
\begin{align*}
\delta_x(z):=\begin{cases} 1  &,\  z=x\\ 0 &,\ \mbox{otherwise.} \end{cases}
\end{align*}
Further, for $p\in[0,1]$, define a probability measure $\mu_x^p:=P\delta_x$, that is,
\begin{align*}
\mu_x^p(z):=\begin{cases} p  &,\  z=x\\ \frac{1-p}{d_x} &,\ z\sim x \\ 0 &,\ \mbox{otherwise} \end{cases}
\end{align*}
\end{definition}

The Ollivier's Ricci curvature (with idleness $p$) is defined at a pair of (different) vertices $x,y\in V$ as 
$$K_p(x,y) := 1-\frac{W_1(\mu_x^p, \mu_y^p)}{d(x,y)}.$$

%Equivalently, the definition can be written as $$W_1(\mu_x^p, \mu_y^p) = d(x,y) (1-K_p(x,y)),$$

The motivation behind the definition of this curvature notion comes from the estimation \eqref{avg_dist}.  The average distance $\overline{d}(B_r(x),B_r(y))$ can be realized as $W_1(\mu_x^p, \mu_y^p)$, and then the term $K_p(x,y)$ is essentially approximated to the Ricci term (up to some constant factor).

For different values of idleness $p$, the distribution $\mu^p_x$ looks differently around $x$. For example, when $p=0$, $\mu^p_x$ resembles a uniformly distributed sphere, while for $p=\frac{1}{d_x+1}$, $\mu^p_x$ resembles a uniformly distributed ball (see Figure \ref{fig:idleness})
 
\begin{figure}[h] 
\centering
\includegraphics[width= 14cm]{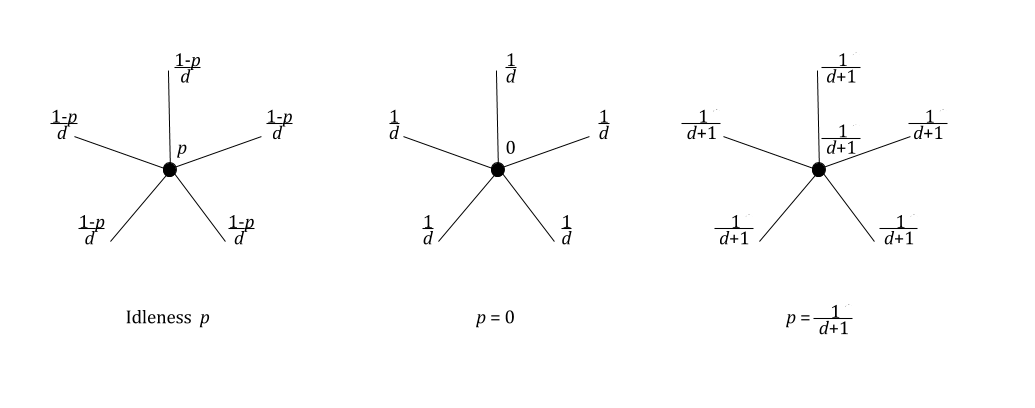}
\caption{Distribution $\mu_x^p$ with idleness $p$, $0$, $\frac{1}{d+1}$, respectively}
\label{fig:idleness}
\end{figure}

When $p=1$, $W_1(\mu_x^p, \mu_y^p)=W_1(\mu_x, \mu_y)=d(x,y)$ which implies $K_1= 0$ identically. Further, Lin-Lu-Yau \cite{LLY} introduced the curvature notion $$K_{LLY} (x,y) := \lim\limits_{p\rightarrow 1^-} \frac{K_p(x,y)}{1-p},$$ which a further insight has been proved in \cite{BCLMP} and \cite{CK} that certainly, 
\begin{align} \label{eqn: scaleLLY}
\frac{K_p(x,y)}{1-p}=K_{LLY} (x,y)
\end{align}
holds for all $x,y\in V$ and all $p\in [\frac{1}{2},1)$.\\

Let us provide an example of how to calculate $K_p(x,y)$, in case $x$ and $y$ are neighbors.
\begin{example} \label{ex: Kcalc}
%First, observe that $W_1(\mu_x^p, \mu_y^p) \le 3$ because one can always transport a portion of mass from any neighbor of $x$ to any neighbor of $y$ through $x$ and $y$ in a total distance at most $3$. Moreover, this transportation can be improved if there are shortcuts between neighbors of $x$ and neighbors of $y$, rather than going through $x$ and $y$. 

Given a graph $G$ as shown in Figure \ref{fig:ollivierpic}. We will consider the transport problem from $\mu_x^p$ (where masses distributed at $x$ for $p$ unit and at $u,w,y$ for $\frac{1-p}{3}$ unit each) to $\mu_y^p$ (where masses distributed at $y$ for $p$ unit and at $x,v,z$ for $\frac{1-p}{3}$ unit each). In consideration of two possible cases, depending on the value of idleness $p$ whether $p\le\frac{1-p}{3}$ or not, we construct in each case a transport plan from $\mu_x^p$ to $\mu_y^p$ that we claim to be optimal.

\begin{figure}[h!] 
\centering
\begin{tikzpicture} [scale=0.8]
\draw[thick] (0,0) node [left] {$x$} 
--(2,0) node [right] {$y$} 
--(2,2) node [right] {$v$} 
--(0,2) node [left] {$u$}
--(0,0)
--(-0.5,-1.93) node [left] {$w$}
--(1,-3) node [below] {$a$}
--(2.5,-1.93) node [right] {$z$}
--(2,0);
\end{tikzpicture}
\caption{Graph $G$ for Example \ref{ex: Kcalc}}
\label{fig:ollivierpic}
\end{figure}
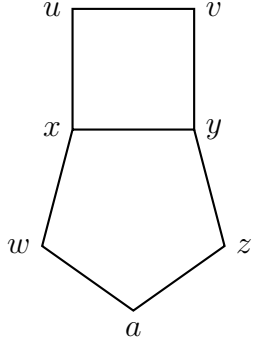

For $p\in[0,\frac{1}{4}]$ (i.e. $p\le \frac{1-p}{3}$), define $\pi_1:V\times V\rightarrow [0,1]$ to be
\begin{align*}
&\pi_1(u,v):=\frac{1-p}{3} &\pi_1(w,x):=\frac{1-p}{3}-p \\ &\pi_1(w,z):=p &\pi_1(y,z):=\frac{1-p}{3}-p
\end{align*}
and to take value zero everywhere else.

One can check that $\pi_1\in \prod(\mu_x^p,\mu_y^p)$ and that $$\textup{cost}(\pi_1)=1\Big(\frac{1-p}{3}\Big)+1\Big(\frac{1-p}{3}-p\Big)+2\Big(p\Big)+1\Big(\frac{1-p}{3}-p\Big)=1-p.$$

In the other case when $p\in[\frac{1}{4},1]$, define $\pi_2:V\times V\rightarrow [0,1]$ to be
\begin{align*}
\pi_2(u,v):=\frac{1-p}{3}\quad\quad \pi_2(x,y):=p-\frac{1-p}{3}\quad\quad  \pi_2(w,z):=\frac{1-p}{3}
\end{align*}
and to take value zero everywhere else.

Similarly, one can also check that $\pi_2\in \prod(\mu_x^p,\mu_y^p)$ and that $$\textup{cost}(\pi_2)=1\Big(\frac{1-p}{3}\Big)+1\Big(p-\frac{1-p}{3}\Big)+2\Big(\frac{1-p}{3}\Big)=\frac{2+p}{3}.$$

By definition of $W_1$, we obtain an upper bound 
\begin{align} \label{eqn: exW_1bound}
W_1(\mu_x^p,\mu_y^p) \le 
\begin{cases}
1-p &\textup{ if } p\in[0,\frac{1}{4}]\\
\frac{2+p}{3} &\textup{ if } p\in[\frac{1}{4},1]
\end{cases}
\end{align}

In order to prove that this upper bound is in fact an equality, we need candidates for optimal Kantorovich potentials. Construct
two $1$-Lipschitz functions $\Phi_1,\Phi_2: V \rightarrow \mathbb{R}$ by assigning their values as
\begin{align*}
\Phi_1(w)=2;\quad \Phi_1(u)=\Phi_1(x)=\Phi_1(y)=1; \quad \Phi_1(v)=\Phi_1(z)=0,  
\end{align*}
and
\begin{align*}
\Phi_2(u)=\Phi_2(w)=2;\quad \Phi_2(x)=\Phi_2(v)=1; \quad \Phi_2(y)=\Phi_2(z)=0,  
\end{align*}
which can be checked that they satisfy $1$-Lipschitz condition on the set of vertices \{x,y,u,v,w,z\}. It is also crucial to remark that one can always extend the $1$-Lipschitz condition onto all other vertices whose values have not yet been assigned. In this example, we (must) assign $\Phi_1(a)=1$ and $\Phi_2(a)=1$. We then obtain a $1$-Lipschitz $\Phi:V\rightarrow \mathbb{R}$, and therefore
\begin{align*}
\sum\limits_{\xi \in V} \Phi_1(\xi)\bigg(\mu_x^p(\xi)-\mu_y^p(\xi)\bigg)
&=2\Big(\frac{1-p}{3}\Big)+1\Big(\frac{1-p}{3}\Big)+1\Big(p-\frac{1-p}{3}\Big)+1\Big(\frac{1-p}{3}-p\Big)\\&=1-p,
\end{align*}
and
\begin{align*}
\sum\limits_{\xi \in V} \Phi_2(\xi)\bigg(\mu_x^p(\xi)-\mu_y^p(\xi)\bigg)&=2\Big(\frac{1-p}{3}\Big)+2\Big(\frac{1-p}{3}\Big)+1\Big(p-\frac{1-p}{3}\Big)+1\Big(0-\frac{1-p}{3}\Big)\\&=\frac{2+p}{3}.
\end{align*}

Therefore Duality Theorem gives $W_1(\mu_x^p,\mu_y^p)\ge \max \{1-p,\frac{2+p}{3}\}$. Together with the upper bound \eqref{eqn: exW_1bound}, we can then conclude that 
\begin{align*} 
W_1(\mu_x^p,\mu_y^p) =
\begin{cases}
1-p &\textup{ if } p\in[0,\frac{1}{4}]\\
\frac{2+p}{3} &\textup{ if } p\in[\frac{1}{4},1]
\end{cases}
\end{align*}
and that
\begin{align*}
K_p(x,y)=1-W_1(\mu_x^p,\mu_y^p) = 
\begin{cases}
p &\textup{ if } p\in[0,\frac{1}{4}]\\
\frac{1-p}{3} &\textup{ if } p\in[\frac{1}{4},1].
\end{cases}
\end{align*}
\end{example}

\section{Reformulation in Linear Optimization} \label{sect:LP}
Let $\mu=\mu_x^p$ and $\nu=\mu_y^p$. The Wasserstein distance $W_1(\mu,\nu)$ as defined in \eqref{W1} can be reformulated in a Linear Optimization (or Linear Programming) problem as follows. Let $x_1=x$ and $x_2,...,x_m$ be all neighbors of $x$. Similarly, let $y_1=y$ and $y_2,...,y_n$ be all neighbors of $y$. 

A transport plan $\pi\in \prod(\mu,\nu)$ may be represented as an $(mn)$-column vector $$\pi=\Big(\pi(x_1,y_1),....,\pi(x_m,y_n)\Big)^T,$$ since $\pi$ takes value zero at any other point $(z,w)\not=(x_i,y_j)$. Also, define the cost-function vector to be a constant $(nm)$-column vector $$d:=\Big(d(x_1,y_1),....,d(x_m,y_n)\Big)^T.$$
As in the definition, $$W_1(\mu,\nu)= \min\limits_{\pi\ge 0} \sum\limits_{i=1}^m \sum\limits_{j=1}^n d(x_i,y_j)\pi(x_i,y_j)$$ subjects to the constraints
\begin{align*}
\sum\limits_{j=1}^m \pi(x_i,y_j)=\mu(x_i) \textup{ for all } i\\
\sum\limits_{i=1}^n \pi(x_i,y_j)=\nu(y_j) \textup{ for all } j
\end{align*}

Concatenate $\mu(x_i)$'s and $\nu(y_j)$'s into a constant $(m+n)$-column vector: $$b=\Big(\mu(x_1),...,\mu(x_m),\nu(y_1),...,\nu(y_n)\Big)^T.$$ The above constraints can then be read as $A \pi = b$ for some constant matrix $A$ of dimension $(m+n)\times (mn)$ matrix $A$ whose entries are $0$'s and $1$'s.

Therefore $W_1(\mu,\nu)$ is a solution to the primal problem (P): $$\min\limits_{\pi\ge 0} d^T\pi \textup{ subjects to } A\pi=b,$$ and its dual problem (D) can be written as $$\max\limits_{\Phi\in\mathbb{R}^{n+m}} b^T\Phi \textup{ subjects to } A^T\Phi\le d,$$
where $\Phi\in\mathbb{R}^{n+m}$ represents the values of $\Phi(x_i)$'s and $\Phi(y_j)$'s, and the constraint $A^T\Phi\le d$ encodes $1$-Lipschitz condition of $\Phi$ among the vertices $x_i$'s and $y_j$'s (and the $1$-Lipschitz condition can then be extended among all vertices in $V$).

The strong duality theorem from Linear Programming then asserts that the solutions of (P) and (D) coincide, which is essentially the statement of Kantorovich Duality Theorem \ref{duality}.

One more crucial aspect from Linear Programming is the Complimentary Slackness theorem: if $\pi^*$ and $\Phi^*$ give optimal solutions to the above (P) and (D) respectively, then $\pi^{*T}(d-A^T\Phi^*)=0$. It can be written equivalently as follows.
\begin{theorem}[Complementary Slackness]
Let $\pi^*$ and $\Phi^*$ be an optimal transport plan and an optimal Kantorovich potential with respect to $W_1(\mu,\nu)$. For any $x_i,y_j$ such that $\pi^*(x_i,y_j)>0$, then $$\Phi^*(x_i)-\Phi^*(y_j)=d(x_i,y_j).$$
\end{theorem}

\section{Ollivier Bonnet-Myers}
In this section, we prove of Bonnet-Myers theorem in the sense of Ollivier's Ricci curvature, which is fairly straightforward (compared to the one in Riemannian geometry or the one in Bakry-\'Emery). In addition, we provide a proof of Lichnerowicz theorem, which is referred to the proof in \cite[Theorem 4.2]{LLY}.

First, we introduce an important lemma, which says that the curvatures between two neighbors give the lower bound for the curvature globally.

\begin{lemma} \label{globalk} Let $G=(V,E)$ be a graph, and $p\in [0,1)$. If $K_p(x,y) \ge K$ holds for all neighboring pairs $x\sim y$, then $K_p(x,y)\ge K$ for all $x,y\in V$.
\end{lemma}
\begin{proof}
Let $L=d(x,y)$ and $x=x_0\sim x_1 \sim...\sim x_L=v$ be a shortest path from $x$ to $y$. By the assumption, $W_1(\mu_{x_i}^p, \mu_{x_{i+1}}^p) =1-K_p(x_i,x_{i+1}) \ge 1-K$ for all $x_i\sim x_{i+1}$. The metric property on $W_1$ then gives
$$W_1(\mu_x^p, \mu_y^p) \le \sum\limits_{i=0}^{L-1} W_1(\mu_{x_i}^p, \mu_{x_{i+1}}^p) \le L(1-K)$$
\noindent and therefore $K_p(x,y) \ge K$.
\end{proof}

\begin{theorem} [O Bonnet-Myers] \label{thm: OBM}
Let $G=(V,E)$ be a connected graph, and $p\in [0,1)$. Assume $K_p(x,y) \ge K > 0$ for all $x\sim y$. Then $G$ is finite and $$\textup{diam } G \le \frac{2(1-p)}{K}$$
\end{theorem}

\begin{proof}
Consider arbitrary vertices $x,y\in V$, and let $L=d(x,y)$. By Lemma \ref{globalk}, the assumption implies $W_1(\mu_x^p, \mu_y^p)\le L(1-K)$. Moreover, the metric $W_1$ gives
\begin{align*}
L=W_1(\delta_x, \delta_y) &\stackrel{\triangle}{\le}W_1(\delta_x, \mu_x^p)+W_1(\mu_x^p, \mu_y^p)+W_1(\mu_y^p, \delta_y)\\
&=2(1-p)+W_1(\mu_x^p, \mu_y^p)\\
&\le 2(1-p)+L(1-K)
\end{align*}
yielding $L\le \frac{2(1-p)}{K}$ as desired. Here we used the fact that $W_1(\delta_x,\mu^p_x)=1-p$ because a mass (of $1-p$ unit in total) is transported by one unit distance from $x$ to its neighbors.
\end{proof}

Note that in Lemma \ref{globalk} and Theorem \ref{thm: OBM}, the $p$-idleness curvature $K_p$ may be replaced Lin-Lu-Yau curvature $K_{LLY}$ (and the diameter of the graph $G$ is bounded by $\textup{diam } G \le \frac{2}{K})$. This is due to the relation \eqref{eqn: scaleLLY}.

\begin{theorem}[O Lichnerowicz]
Let $G=(V,E)$ be a finite connected graph. Assume there exists $K>0$ such that $K_{LLY}(x,y) \ge K$ for all $x\sim y$. Then the first nonzero eigenvalue $\lambda_1 \ge K$.
\end{theorem}

\begin{proof}
For any number $p \in [\frac{1}{2},1)$, $$\frac{K_p(x,y)}{1-p}=K_{LLY}(x,y) \ge K$$ 
holds for all $x\sim y$, and also holds for all $x,y\in V$, because of Lemma \ref{globalk} and relation \eqref{eqn: scaleLLY}. Consider an average operator $M_p:C(V)\rightarrow C(V)$ defined as

\begin{equation}
M_p(f)(x) := \sum\limits_{z\in V} \mu_x^p(z)f(z)
\end{equation}
or equivalently, 
\begin{align*}
M_p f(x) &= pf(x)+\sum\limits_{z:z\sim x} \frac{1-p}{d_x}f(z)\\
&= f(x) - (1-p)\Delta f(x)
\end{align*}

\noindent Let $f_1$ be an eigenfunction satisfying $\Delta f_1 = \lambda_1 f_1$. Hence 

\begin{equation} \label{mpf}
M_p f_1(x)=(1-(1-p)\lambda_1)f_1(x).
\end{equation}

\noindent Consider $$\ell := \max\limits_{x,y\in V} \frac{|f_1(x)-f_1(y)|}{d(x,y)},$$ attaining a maximum at $x=x_1$, $y=y_1$. Note that $\ell\not=0$; otherwise, $f_1$ is constant giving $\Delta f_1=0$.  We have

\begin{equation} \label{mp1}
\bigg|M_p\frac{f_1}{\ell}(x_1)-M_p\frac{f_1}{\ell}(y_1) \bigg| \stackrel{\eqref{mpf}}{=} |1-(1-p)\lambda_1|\cdot d(x_1,y_1)
\end{equation}
On the other hand, note that $\frac{f_1}{\ell}$ is $1$-Lipschitz function. Hence by Duality Theorem, we have
\begin{align} \label {mp2}
\bigg|M_p\frac{f_1}{\ell}(x)-M_p\frac{f_1}{\ell}(y) \bigg| &=\bigg|\sum\limits_{z\in V} \frac{f(z)}{\ell}\big(\mu_x^p(z)-\mu_y^p(z)\big)\bigg| \nonumber \\
&\stackrel{\ref{duality}}{\le} W_1(\mu_x^p,\mu_y^p)\nonumber \\
&=\big(1-K_p(x,y)\big)\cdot d(x,y) \nonumber \\
&\le \big(1-(1-p)K\big)\cdot d(x,y)
\end{align} \noindent holds for all $x,y$. After substituting $x=x_1$, $y=y_1$  into inequality \eqref{mp2} and comparing to equation \eqref{mp1}, we finally obtain $\lambda_1\ge K$.
\end{proof}

\chapter*{Acknowledgements}
\thispagestyle{empty}

I would like to express my deepest appreciation to my supervisor, Professor Norbert Peyerimhoff, who invites me to this field of research and gives suggestions throughout my dissertation. I would also like to thank David Cushing for his intelligible introduction lectures of the discrete curvature notions, as well as his creative research ideas in this topic.\\

I would like to thank Department of Mathematical Sciences, Durham University, and especially Dr. Wilhelm Klingenberg for lectures that prepare me with good background knowledge in Riemannian geometry.\\

Lastly, I would like to thank my family and friends who always support me, and Royal Thai government who provides a scholarship for my current study.  

\clearpage
%\printbibliography

%\chapter*{Assertion of Authorship}
%\thispagestyle{empty}
%This dissertation is my own work, except where explicitly acknowledged by the giving of a reference.\\\\\\Name: SUPANAT KAMTUE
%\\\\\\Signature and date: 
\end{document}